\documentclass[psamsfonts]{amsart}

\usepackage{amssymb,amsfonts}
\usepackage[all,arc]{xy}
\usepackage{enumerate}
\usepackage{enumitem}
\usepackage{mathrsfs}
\usepackage{hyperref}
\usepackage{cleveref}
\usepackage{graphicx}
\usepackage{mathtools}
\usepackage[dvipsnames]{xcolor}
\usepackage[T1]{fontenc}
\usepackage[utf8]{inputenc}

\newtheorem{thm}{Theorem}[section]
\newtheorem{cor}[thm]{Corollary}
\newtheorem{prop}[thm]{Proposition}
\newtheorem{lem}[thm]{Lemma}

\newtheorem{obs}[thm]{Observation}

\theoremstyle{definition}
\newtheorem{defn}[thm]{Definition}

\theoremstyle{remark}
\newtheorem{rem}[thm]{Remark}

\makeatletter
\let\c@equation\c@thm
\makeatother
\numberwithin{equation}{section}

\usepackage{graphicx}

\title[Asymptotically Large Free Semigroups]{Asymptotically large free semigroups in Zariski dense discrete subgroups of Lie groups}
\author{Aleksander Skenderi}
\address{Department of Mathematics, Jagiellonian University, Kraków, Poland}
\email{aleksander.skenderi@uj.edu.pl}
\begin{document}
\begin{abstract}
Let $G$ be a connected algebraic semisimple real Lie group with finite center and no compact factors, and let $\Gamma$ be a Zariski dense discrete subgroup of $G$. We show that $\Gamma$ contains free, finitely generated subsemigroups whose critical exponents are arbitrarily close to that of $\Gamma$. Furthermore, these subsemigroups are Zariski dense in $G$ and $P$-Anosov in the sense of Kassel--Potrie \cite{KP}. This shows that no gap phenomenon holds for critical exponents of discrete subsemigroups of Lie groups, which is in contrast with Leuzinger's critical exponent gap theorem for infinite covolume discrete subgroups of Lie groups with Kazhdan's property (T), proven in 2003 \cite{Leu}.
\par 
As an application, we prove that the critical exponent is lower semicontinuous in the Chabauty topology, in the following sense: if a sequence of Zariski dense discrete subgroups $\{\Gamma_{n}\}$ of $G$ converges in the Chabauty topology to a Zariski dense discrete subgroup $\Gamma$, then $\liminf_{n \to \infty} \delta(\Gamma_{n}) \geq \delta(\Gamma)$.
\end{abstract}
\dedicatory{Dedicated to the memory of my father, Vasil Skenderi}
\maketitle  
\tableofcontents
\section{Introduction}
\subsection{Motivation and Statements of Main Results}
Let $G$ be a connected algebraic semisimple real Lie group with finite center and no compact factors, and let $X$ denote the symmetric space of $G$. Given a discrete subgroup $\Gamma < G$, an important quantity associated with the action of $\Gamma$ on the symmetric space $X$ is the \emph{critical exponent of} $\Gamma$. To define this object, fix a $G$-invariant metric $d_{X}$ induced from a Riemannian metric on $X$, and fix a basepoint $o \in X$. The critical exponent of $\Gamma$, denoted by $\delta(\Gamma)$, is the abscissa of convergence of the \emph{Poincar\'e series}
\begin{align*}
Q_{\Gamma}(s) := \sum_{\gamma \in \Gamma} e^{-sd_{X}(o,\gamma o)},
\end{align*} that is,
\begin{align*}
\delta(\Gamma) := \inf \{s > 0 : Q_{\Gamma}(s) < \infty \} \in [0, \infty].
\end{align*} Equivalently, the critical exponent of $\Gamma$ is given by 
\begin{align*}
\delta(\Gamma) = \limsup_{T \to \infty} \frac{1}{T} \log \# \{ \gamma \in \Gamma : d_{X}(o, \gamma o) \leq T \},
\end{align*} and thus it measures the exponential growth rate of the $\Gamma$-orbits in $X$. Since $\Gamma$ acts on $X$ by isometries, the critical exponent is independent of the choice of basepoint $o \in X$, hence is well-defined. 
\par 
A natural question to ask is whether one can compute the critical exponent. In the case when $\Gamma$ is a \emph{lattice in} $G$ (that is, the locally symmetric space $\Gamma \backslash X = \Gamma \backslash G / K$ has finite volume), the critical exponent of $\Gamma$ coincides with the \emph{volume growth entropy of} $X$, which is defined by 
\begin{align*}
h_{\mathrm{vol}}(X) := \lim_{R \to \infty} \frac{1}{R} \log \mathrm{vol}(B_{R}(o)).
\end{align*} Here $B_{R}(o)$ is the ball of radius $R$ centered at $o \in X$ and $\mathrm{vol}$ denotes the Riemannian volume on $X$. We remark that the volume growth entropy is independent of the choice of basepoint $o \in X$.
\par 
In the case when $\Gamma$ is a infinite covolume discrete subgroup of $G$ (that is, $\Gamma \backslash X$ has infinite volume), the critical exponent may take different values depending on the particular subgroup. For instance, Sullivan \cite{Sul1} constructed a sequence of convex cocompact discrete subgroups of $\mathrm{SL}(2, \mathbb{C}) \cong \mathrm{Isom}(\mathbb{H}_{\mathbb{R}}^{3})$ whose critical exponents are arbitrarily close to $2$ (the value attained by lattices). See also section 4 of \cite{KK} for such examples in real hyperbolic spaces of all dimensions and Theorem A of \cite{DL} for examples in complex hyperbolic spaces of dimensions $2$ and $3$.
\par 
However, when the Lie group $G$ has Kazhdan's property (T), the situation is considerably different. This was first noticed by Corlette \cite{Cor}, who established a remarkable gap theorem for the critical exponents of infinite covolume discrete subgroups of isometries of the quaternionic and octonionic hyperbolic spaces. Denote by $\mathbb{H}_{\mathbb{H}}^{n}$ the $n$-dimensional quaternionic hyperbolic space and $\mathbb{H}_{\mathbb{O}}^{2}$ the Cayley hyperbolic plane (also known as the octonionic projective plane). These are connected, contractible, negatively curved Riemannian manifolds with normalized sectional curvatures between $-4$ and $-1$. The simple Lie groups of real rank one Sp$(n,1)$ and $\mathrm{F}_{4}^{-20}$ are the orientation-preserving isometry groups of $\mathbb{H}_{\mathbb{H}}^{n}$ and $\mathbb{H}_{\mathbb{O}}^{2}$, respectively. Corlette's renowned gap theorem states the following.
\begin{thm} [Corlette, Theorem 4.4 of \cite{Cor}]
\label{CorletteGapThm} ~\
\begin{itemize}
    \item[(1)] If $\Gamma \subset \mathrm{Sp}(n,1)$, $n \geq 2$, is a discrete subgroup, then $\delta(\Gamma) = 4n+2$ or $\delta(\Gamma) \leq 4n$. Moreover, $\delta(\Gamma) = 4n+2$ if and only if $\Gamma$ is a lattice.
    \item[(2)] If $\Gamma \subset \mathrm{F}_{4}^{-20}$ is a discrete subgroup, then $\delta(\Gamma) = 22$ or $\delta(\Gamma) \leq 16$. Moreover, $\delta(\Gamma) = 22$ if and only if $\Gamma$ is a lattice. 
\end{itemize}
\end{thm}
Inspired by Corlette's result, Leuzinger later showed in \cite{Leu} that a similar gap phenomenon holds for any infinite covolume discrete subgroup of a semisimple real Lie group $G$ having Kazhdan's property (T) (this is equivalent to $G$ having no simple factors which are isogenous to SO$(n,1)$ or SU$(n,1)$, the isometry groups of the real and complex hyperbolic spaces, respectively).
\begin{thm} [Leuzinger, Main Theorem (Dichotomy) of \cite{Leu}]
\label{Leuzingergapphenomenon}
Let $G$ be a connected semisimple real Lie group with finite center, with no compact factors, and with Kazhdan's property (T). Let $\Gamma$ be a discrete subgroup of $G$, and let $X$ be the symmetric space of $G$. There exists a constant $\epsilon = \epsilon(G) > 0$, depending only on $G$ and not on $\Gamma$, such that the following holds:
\begin{itemize}
    \item[(1)] The discrete subgroup $\Gamma$ is a lattice in $G$ if and only if
    \begin{align*}
    \delta(\Gamma) = h_{\mathrm{vol}}(X).
    \end{align*}
    \item[(2)] The discrete subgroup $\Gamma$ has infinite covolume in $G$ if and only if
    \begin{align*}
    \delta(\Gamma) \leq h_{\mathrm{vol}}(X) - \epsilon.
    \end{align*}
\end{itemize}
\end{thm}
In previous work \cite{S}, the author showed that the counterpart to Corlette's Theorem \ref{CorletteGapThm} does not hold for the class of \emph{discrete subsemigroups} of Sp$(n,1)$ or $\mathrm{F}_{4}^{-20}$ by showing that, for any lattice in Sp$(n,1)$ or $\mathrm{F}_{4}^{-20}$, there exist finitely generated free subsemigroups of critical exponent arbitrarily close to, but strictly less than, that of the lattice; this result also follows from earlier work of Yang \cite{Y} who used different methods. Inspired by the seminal work of Bishop--Jones \cite{BJ}, the author established this result by working in the broader context of convergence groups with expanding coarse-cocycles, introduced earlier by Blayac--Canary--Zhu--Zimmer in \cite{BCZZ1}. Since any discrete subgroup of a rank one Lie group acts as a convergence group on its limit set, the author was able to apply his general result to the specific cases of lattices in Sp$(n,1)$ or $\mathrm{F}_{4}^{-20}$. Furthermore, these results also apply to the class of discrete subgroups of higher-rank Lie groups, known as \emph{transverse groups}. Roughly speaking, these are discrete subgroups with a well-defined limit set in an appropriate flag variety of $G$, with the additional property that the action of the discrete subgroup on its limit set is a convergence group action. 
\par 
While the class of transverse groups is broad enough to include many interesting types of groups (such as all discrete subgroups of rank one Lie groups, Anosov and relatively Anosov groups in higher rank Lie groups, and their subgroups), they are still very special in the sense that a generic discrete subgroup in higher rank need not act as a convergence group on its limit set. Hence, the author's results in \cite{S} do not apply for arbitrary discrete subgroups in higher rank Lie groups.
\par 
The purpose of this paper is to show that, under the very mild hypothesis that the discrete subgroup $\Gamma$ of $G$ is \emph{Zariski dense} (for example, any lattice in $G$ is Zariski dense by Borel's Density Theorem), we can find free, finitely generated, Zariski dense subsemigroups of $\Gamma$ whose critical exponents are arbitrarily close to that of $\Gamma$. Postponing some definitions, our main result is the following:
\begin{thm} [Theorem \ref{MainTheorem2}]
\label{MainTheorem1}
Let $G$ be a connected algebraic semisimple real Lie group with finite center and no compact factors, and let $\Gamma < G$ be a Zariski dense discrete subgroup. For every $0 < \delta < \delta (\Gamma)$, $\epsilon > 0$ sufficiently small, and $B \geq 1$, there exists a free, finitely generated subsemigroup $\Omega = \Omega_{\delta, \epsilon,B} \subset \Gamma$ with the following properties:
\begin{itemize}
    \item[(1)] Every element of $\Omega$ is either $\epsilon$-contracting or $2 \epsilon$-contracting.
    \item[(2)] The semigroup $\Omega$ is Zariski dense in $G$. 
    \item[(3)] The critical exponent of $\Omega$ satisfies 
    \begin{align*}
    \delta(\Omega) \geq \delta.
    \end{align*}
    \item[(4)] The semigroup $\Omega$ is $P$-Anosov. In fact,
    \begin{align*}
    \min_{\alpha \in \Delta} \alpha (\kappa (g)) \geq B |g|_{S},
    \end{align*} for all $g \in \Omega$, where $|\cdot|_{S}$ denotes the word-length with respect to the finite generating set $S$ that freely generates $\Omega$.
\end{itemize}
\end{thm}
An immediate consequence of our result is that the counterpart to Leuzinger's Theorem \ref{Leuzingergapphenomenon} does not hold for the class of Zariski dense discrete \emph{subsemigroups} of $G$. 
\subsection{Ideas of the Proof} 
The proof of Theorem \ref{MainTheorem1} is rather technical at times, so for the reader's convenience we will attempt to convey some of the main ideas and strategies behind it.
\par 
In general, estimating the critical exponent of a discrete subsemigroup $\Omega \subset \Gamma$ is a challenging task. Indeed, it requires having some understanding of the real numbers $s > 0$ for which the Poincar\'e series
\begin{align*}
Q_{\Omega}(s) := \sum_{\gamma \in \Omega} e^{-s d_{X}(o, \gamma o)}
\end{align*} diverges. In particular, if there are many non-trivial relations among the elements of the semigroup $\Omega$, then there is no apparent approach to evaluating this series. However, if the semigroup $\Omega$ is \emph{freely generated} by some finite generating set $S$, then, writing $S^{m}$ for all the words in $\Omega$ of word length $m \geq 1$, we have
\begin{align*}
Q_{\Omega}(s) = \sum_{\gamma \in \Omega} e^{-s d_{X}(o, \gamma o)} = \sum_{m=1}^{\infty} \sum_{\gamma \in S^{m}} e^{-s d_{X}(o,\gamma o)}.
\end{align*} Thus to compute $Q_{\Omega}(s)$, it suffices to understand the sums
\begin{align*}
\sum_{\gamma \in S^{m}} e^{-s d_{X}(o, \gamma o)}, 
\end{align*} for all $m \geq 1$. Fix $0 < \delta < \delta(\Gamma)$. The point now is that, since the metric $d_{X}$ is $G$-invariant, one can reduce the problem of showing that the above series diverges at $s = \delta$ (i.e., that $\delta(\Omega) \geq \delta$) to showing that the finite generating set $S$ which freely generates $\Omega$ has the further property that the finite sum
\begin{align*}
\sum_{\zeta \in S} e^{-\delta d_{X}(o, \zeta o)}
\end{align*} is sufficiently large. Therefore, we are in search of a finite subset $S \subset \Gamma$ having the following properties:
\begin{enumerate}
    \item\label{item:1} The sum $\sum_{\zeta \in S} e^{-\delta d_{X}(o, \zeta o)}$ is sufficiently large.
    \item\label{item:2} There are no relations when multiplying positive powers of elements of $S$.
    \item\label{item:3} The Zariski closure of the semigroup $\Omega$ generated by $S$, which is a subgroup of $G$, (see for instance Lemma 6.15 of \cite{BQ1} for a proof in the linear case) is not contained in any proper closed connected subgroup of $G$.
\end{enumerate} Using Quint's growth indicator function, we are able to find (many) cones in the symmetric space $X$ of $G$ containing a large number of elements of the $\Gamma$-orbit of $o \in X$; the precise statement is Corollary \ref{LargeGrowthInd}. We will look for the subset $S$ among these elements in such a way that items (\ref{item:2}) and (\ref{item:3}) above are also satisfied. 
\par 
One can think of item (\ref{item:2}) as saying that the orbit of the base point $o \in X$ under the semigroup $\Omega$ generated by $S$ is a tree. Indeed, the vertices are the orbit points $\Omega \cdot o$ and two vertices $\eta o$ and $\gamma o$ are connected by an edge if and only if $\eta = \gamma \zeta$ for some $\zeta \in S$. To establish this tree-like structure, the idea is to perform a similar construction as in our earlier work \cite{S}, which was in turn heavily inspired by the seminal work of Bishop--Jones \cite{BJ}. A fundamental difficulty in this regard is finding a good notion of ``shadows'' of elements of $\Omega$ in the Furstenberg boundary $\mathcal{F} = G/P$. Precisely, we want our shadows to satisfy the following: for all $m \geq 1$, if $\eta = \gamma \zeta$ where $\gamma \in S^{m}$ and $\zeta \in S$, then some shadow of $\eta$ is contained in some shadow of $\gamma$. Furthermore, if $\eta' = \gamma \zeta'$ for some $\zeta' \in S \smallsetminus \{\zeta\}$, then the shadows of $\eta$ and $\eta'$ are disjoint. The key difficulty here is that the various notions of shadows currently available in the literature do not (at least to the author) seem to be amenable to such delicate requirements.
\par 
For this reason, we study certain loxodromic elements of semisimple Lie groups, which we call $\epsilon$-$\emph{contracting elements}$ (see Definition \ref{epsiloncontractingelement} for the precise definition). We remark that these elements were first introduced by Guichard--Wienhard in their work on Anosov representations; see Definition 5.6 of \cite{GW} (which was itself inspired by the important work of Abels--Margulis--Soifer \cite{AMS}). However, our work appears to be the first to develop many basic properties of these elements and to demonstrate their usefulness in studying arbitrary Zariski dense discrete subgroups of Lie groups. We emphasize that our use of $\epsilon$-contracting elements is motivated by the monumental works of Abels--Margulis--Soifer \cite{AMS} and Benoist \cite{B1}, \cite{B2}, which heavily used the notion of $\epsilon$-proximal elements. These elements are well-behaved under taking products (Proposition \ref{contractingsemigroup}) and have a natural associated notion of ``shadow'' in the Furstenberg boundary $\mathcal{F}$ (Definition \ref{contractingshadow}), which is similarly very well-behaved (Proposition \ref{shadowinclusion} and Lemma \ref{contractsymspaceinclusion}). 
\par 
For this definition to be useful in our context, we need to show that sufficiently many of the elements we found using Quint's growth indicator function are $\epsilon$-contracting and, moreover, are ``well-positioned'' with respect to each other so as to apply the aforementioned results. It is for this reason that sets of the form $\Gamma_{\mathcal{C},x,y,n,\epsilon}$, introduced in Section \ref{asymptoticphenomena}, play such a central role in this paper. Additionally, by establishing that these sets are Zariski dense in $G$ 
(see Proposition \ref{ZariskidensityCartanset} and Remark \ref{afortiori}) we will be able to deduce item (\ref{item:3}) above. 
\par 
To conclude this discussion, we mention that another key feature of our proof is that it avoids the use of Tits representations \cite{Ti}, except insofar as we use the results of Benoist \cite{B1}, \cite{B2} and Quint \cite{Q1}, \cite{Q2}, which make use of these representations. Tits representations have been used in many important works (for instance, \cite{AMS}, \cite{B1}, \cite{B2}, \cite{BQ1}, \cite{Q1}, and \cite{Q2}) to embed $\mathcal{F}$ into a product of projective spaces associated with these representations and then study the $\Gamma$-action on $\mathcal{F}$ via the action of its linear representations on the product of projective spaces. By instead using $\epsilon$-contracting elements and a ``north-south dynamics'' result about the $G$-action on $\mathcal{F}$ (Proposition \ref{northsouthflagvars}), we are able to study the action of $\Gamma$ on $\mathcal{F}$ without appealing to Tits representations. We hope this approach will have future applications, or potentially lead to simplifications of results currently in the literature. To highlight the fruitfulness of our approach, we discuss three applications to arbitrary Zariski dense discrete subgroups, which use both our main result Theorem \ref{MainTheorem1} and the techniques of its proof.

\subsection{Applications}
The first application is a group-theoretic consequence of Theorem \ref{MainTheorem1} and Leuzinger's Theorem \ref{Leuzingergapphenomenon}. Let $G$ be as in Theorem \ref{MainTheorem1}, assume further that $G$ has Kazhdan's property (T), and let $\Gamma < G$ be a lattice. By Leuzinger's Theorem \ref{Leuzingergapphenomenon}, there exists a constant $\epsilon = \epsilon(G) > 0$ depending only on the Lie group $G$ so that, for any infinite covolume discrete subgroup $\Lambda$ of $G$, we have $\delta(\Lambda) \leq \delta(\Gamma) - \epsilon$. By Theorem \ref{MainTheorem1}, there exists a free, finitely generated subsemigroup $\Omega \subset \Gamma$ with finite generating set $S \subset \Gamma$ so that 
\begin{align*}
\delta(\Gamma) - \epsilon < \delta(\Omega) \leq \delta(\Gamma).
\end{align*} Let $\Delta := \langle S \rangle$ be the group generated by the set $S$. Then $\Delta$ is a discrete subgroup of $G$ contained in $\Gamma$ whose critical exponent satisfies
\begin{align*}
\delta(\Gamma) - \epsilon < \delta(\Omega) \leq \delta(\Delta) \leq \delta(\Gamma).
\end{align*} Leuzinger's Theorem \ref{Leuzingergapphenomenon} therefore implies that $\delta(\Delta) = \delta(\Gamma)$ and $\Delta$ is a lattice with finite index in $\Gamma$. In particular, there is some relation among the elements of $S$ when one is allowed to multiply by negative powers of these elements, even though there is no relation when multiplying together exclusively positive powers of these elements.

The second application, which is more significant, establishes that the critical exponent is lower semicontinuous in the Chabauty topology if the Chabauty limit of a sequence of Zariski dense discrete subgroups is also a Zariski dense discrete subgroup. Precisely, using Theorem \ref{MainTheorem1}, the properties of $\epsilon$-contracting elements that we have developed, and other aspects of the proof Theorem \ref{MainTheorem1}, we show the following.
\begin{thm} [Theorem \ref{lowersemicontinuity2}]
\label{lowersemicontinuity1}
Let $G$ be a connected algebraic semisimple real Lie group with finite center and no compact factors. Let $\{\Gamma_{n}\}$ be a sequence of Zariski dense discrete subgroups of $G$ which converges in the Chabauty topology to a Zariski dense discrete subgroup $\Gamma$ of $G$. Then,
\begin{align*}
\liminf_{n \to \infty} \delta(\Gamma_{n}) \geq \delta(\Gamma).
\end{align*}
\end{thm}
To highlight the significance of this result, we provide a new proof of a result concerning the Chabauty convergence of infinite covolume Zariski dense discrete subgroups in semisimple Lie groups with Kazhdan's property (T); see Corollary \ref{Infinitecovolumeconvergence}. As far as the author is aware, previously known proofs relied on the very deep superrigidity theorems of Margulis \cite{Mar} (in higher rank) and Corlette \cite{Cor2} and Gromov-Schoen \cite{GS} (in the exotic rank one cases). By contrast, our proof only uses Theorem \ref{lowersemicontinuity1} and Theorem \ref{Leuzingergapphenomenon}, thereby providing a new, simpler, proof.
\subsection{Outline of the Paper}
In section \ref{preliminariesonLiegroups}, we recall some standard facts about the structure theory of semisimple Lie groups as well as dynamics on flag varieties. In section \ref{backgroundondiscretesubgroups}, we recall certain asymptotic objects associated to discrete subgroups of Lie groups introduced by Benoist \cite{B2} and Quint \cite{Q1} and provide proofs of certain basic properties of Quint's growth indicator function. In section \ref{contractingelementssection}, we introduce $\epsilon$-contracting elements, their shadows, and establish some properties of these objects that will be relevant in what follows. Section \ref{asymptoticphenomena} is perhaps the most technical part of the paper, where the majority of the preliminary results needed to find the generating sets of our semigroups are established. In section \ref{symspaceshadowssection}, we relate the shadows of $\epsilon$-contracting elements to so-called ``symmetric space shadows,'' which have already been used extensively in the literature. This is needed in order for us to establish the freeness of our semigroup. In section \ref{culminationsection}, we combine the results of the preceding sections to prove Theorem \ref{MainTheorem1}. Lastly, in section \ref{lowersemicontsection} we prove Theorem \ref{lowersemicontinuity1} and use it to give a new proof of Corollary \ref{Infinitecovolumeconvergence}.
\subsection*{Acknowledgements}
The contents of this paper form part of my PhD thesis at the University of Wisconsin-Madison. As such, I want to begin by thanking my advisors Andrew Zimmer and Sebastian Hurtado-Salazar for all their help, not only while I was working on this project, but also throughout my time in graduate school. I am especially grateful to Andrew for listening to all of my half-baked ideas, being very generous with his ideas and time during so many conversations, and for encouraging me to work on this project and strengthen the initial results of this paper. I am very thankful for Sebastian for so many helpful conversations, his interest in my work, and for encouraging me to pursue this project. Moreover, Sebastian suggested that my methods could be used to prove Theorem \ref{lowersemicontinuity1}, which was indeed the case.

As will be clear to the reader, the beautiful works \cite{B1} and \cite{B2} of Yves Benoist have had a great influence on the present paper; it is a pleasure to thank Yves Benoist for his inspiring mathematics.

I am very grateful to Mikołaj Frączyk and Dongryul M. Kim for numerous helpful discussions. In particular, Mikołaj suggested that the Bruhat decomposition may be useful in establishing Lemma \ref{savedbyBruhat}, which indeed turned out to be the case, and Dongryul informed me of Proposition \ref{continuityonviewpoints}, which enabled me to deduce Lemma \ref{diametertozero}. I also thank Fernando Al Assal for nice discussions and listening to several of my inchoate ideas.

This material is based upon work supported by the National Science Foundation under Grant No. DMS-2037851 and Grant No. DMS-2230900. This project was also financed under Dioscuri, a program initiated by the Max Planck Society, jointly managed with the National Science Center in Poland, and mutually funded by the Polish Ministry of Science and Higher Education and the German Federal Ministry of Research, Technology, and Space. 

\subsection*{Personal Declaration} As some edits to this work were made while I was supported by the Dioscuri Grant, I was required, for the purpose of Open Access, to apply a CC BY public copyright license to any Author Accepted Manuscript (AAM) version arising from this submission. However, I do not consent to the contents of this document, either in part or in full, being uploaded to any generative artificial intelligence system.
\par 
While the following is likely evident from the previous sentence, it seems to me worthwhile to make it explicit: no generative artificial intelligence was used in any way, at any point, in the process of creating this manuscript. I have no desire to delegate my birthright of independent thought to a machine (which strikes me as both unwise and unenjoyable, among other things), nor do I have any desire to use the products of people who have, repeatedly, shown a complete lack of care for both the mathematical community and the broader world.
\section{Preliminaries on Semisimple Lie Groups}
\label{preliminariesonLiegroups}
\subsection{Basic structure theory of semisimple Lie groups}
Let $G$ be a connected algebraic semisimple real Lie group without compact factors and with finite center and let $\mathfrak{g}$ denote its Lie algebra. Let $b$ denote the Killing form of $\mathfrak{g}$ and fix a Cartan involution $\tau$ of $\mathfrak{g}$; that is, an involution of $\mathfrak{g}$ for which the bilinear pairing $\langle \cdot, \cdot \rangle$ on $\mathfrak{g}$ defined by $\langle X, Y \rangle := -b(X, \tau (Y))$ is an inner product. Then $\mathfrak{g}$ decomposes as $\mathfrak{g} = \mathfrak{k} \oplus \mathfrak{p}$, where $\mathfrak{k}$ and $\mathfrak{p}$ are the $1$ and $-1$ eigenspaces of $\tau$. The subalgebra $\mathfrak{k}$ is a maximal compact Lie subalgebra of $\mathfrak{g}$ and we denote by $K \subset G$ the maximal compact Lie subgroup of $G$ whose Lie algebra is $\mathfrak{k}$.

Fix a maximal abelian subalgebra $\mathfrak{a} \subset \mathfrak{p}$, known as a \emph{Cartan subspace}, which is unique up to conjugation. The Lie algebra $\mathfrak{g}$ then decomposes as 
\begin{align*}
    \mathfrak{g} = \mathfrak{g}_{0} \oplus \bigoplus_{\alpha \in \Sigma} \mathfrak{g}_{\alpha},
\end{align*} which is called the \emph{restricted root space decomposition} associated to $\mathfrak{a}$; in this decomposition, for $\alpha \in \mathfrak{a}^{*}$, we define
\begin{align*}
    \mathfrak{g}_{\alpha} := \{X \in \mathfrak{g} : [H,X] = \alpha (H)X \ \ \mathrm{for \ all} \ H \in \mathfrak{a} \},
\end{align*} and call
\begin{align*}
    \Sigma := \{ \alpha \in \mathfrak{a}^{*} \smallsetminus \{0\} : \mathfrak{g}_{\alpha} \neq 0 \}
\end{align*} the set of \emph{restricted roots}. Now fix an element $H_{0} \in \mathfrak{a}$ so that $\alpha (H_{0}) \neq 0$ for all $\alpha \in \Sigma$, and let
\begin{align*}
\Sigma^{+} := \{ \alpha \in \Sigma : \alpha (H_{0}) > 0 \} \ \ \mathrm{and} \ \ \Sigma^{-} := - \Sigma^{+}.
\end{align*} Notice that $\Sigma = \Sigma^{+} \sqcup \Sigma^{-}$. We write $\Delta \subset \Sigma^{+}$ for the set of \emph{simple restricted roots}, which, by definition, consists of all the elements of $\Sigma^{+}$ which cannot be written as a non-trivial positive integral linear combination of elements in $\Sigma^{+}$. As $\Sigma$ is an abstract root system on $\mathfrak{a}^{*}$, it follows that $\Delta$ is a basis of $\mathfrak{a}^{*}$ and every $\alpha \in \Sigma^{+}$ is a non-negative integral linear combination of elements in $\Delta$. See for instance Chapter II of Knapp's book \cite{Kn} for more details.
\subsubsection{The Weyl Group, Cartan Projection, and Jordan Projection}
 The \emph{Weyl group} of $\mathfrak{a}$ is given by $\mathcal{W} := N_{K}(\mathfrak{a}) / Z_{K}(\mathfrak{a})$, where $N_{K}(\mathfrak{a}) \subset K$ is the normalizer of $\mathfrak{a}$ in $K$ and $Z_{K}(\mathfrak{a}) \subset K$ is the centralizer of $\mathfrak{a}$ in $K$. The Weyl group is a finite group generated by reflections of $\mathfrak{a}$ (with respect to the inner product $\langle \cdot, \cdot \rangle$) about the kernels of the simple restricted roots in $\Delta$. Hence, $\mathcal{W}$ acts transitively on the set of \emph{Weyl chambers}, which are the closures of the connected components of 
\begin{align*}
    \mathfrak{a} - \bigcup_{\alpha \in \Sigma} \ker \alpha.
\end{align*} We call the Weyl chamber 
\begin{align*}
    \mathfrak{a}^{+} := \{X \in \mathfrak{a} : \alpha (X) \geq 0 \ \ \mathrm{for \ all} \ \alpha \in \Delta \},
\end{align*} the \emph{positive Weyl chamber}. We set $A := \exp (\mathfrak{a})$, $A^{+} := \exp (\mathfrak{a}^{+})$, and $\mathfrak{a}^{++} = \mathrm{int}(\mathfrak{a}^{+})$. In the Weyl group $\mathcal{W}$, there is a unique element $w_{0}$, called the \emph{longest element}, with the property that $w_{0}(\mathfrak{a}^{+}) = - \mathfrak{a}^{+}$. Thus the longest element allows us to define an involution $\iota : \mathfrak{a} \rightarrow \mathfrak{a}$, $H \mapsto - w_{0} \cdot H$, which is called the \emph{opposition involution}. It induces an involution of $\Sigma$ preserving $\Delta$, denoted by $\iota^{*}$, defined by $\iota^{*}(\alpha) = \alpha \circ \iota$ for all $\alpha \in \Delta$. Moreover, if $k_{0} \in N_{K}(\mathfrak{a})$ denotes a representative of the longest element $w_{0} \in \mathcal{W}$, then
\begin{align}
\label{OppInvAct}
\mathrm{Ad}(k_{0}) \mathfrak{g}_{\alpha} = \mathfrak{g}_{-\iota^{*}(\alpha)}
\end{align} for all $\iota \in \Sigma$.
\par 
Let $\kappa : G \rightarrow \mathfrak{a}^{+}$ denote the \emph{Cartan projection}, that is $\kappa (g) \in \mathfrak{a}^{+}$ is the unique element so that 
\begin{align*}
    g = k \exp (\kappa (g)) \ell
\end{align*} for some $k, \ell \in K$. We note that $k, \ell \in K$ need not be unique. Such a decomposition of $g \in G$ is called a $KA^{+}K$ decomposition (see Theorem $7.39$ of \cite{Kn}). Observe that $\iota(\kappa(g)) = \kappa(g^{-1})$ for all $g \in G$. Using the Cartan projection, we can define the map $\lambda : G \rightarrow \mathfrak{a}^{+}$ known as the \emph{Jordan projection} by 
\begin{align*}
\lambda(g) := \lim_{n \to \infty} \frac{\kappa(g^{n})}{n}.
\end{align*}
Let $H$ be a connected real algebraic Lie group, that is, $H$ is not assumed to be semisimple. Such an algebraic group admits a \emph{Levi decomposition}
\begin{align*}
H = L \ltimes R_{u}(H),
\end{align*} where $L$ is a reductive subgroup of $H$ known as the \emph{Levi subgroup} and $R_{u}(H)$ is the unipotent radical of $H$. See Chapter 6 of \cite{OV} for further details. 
\subsubsection{The Symmetric Space of $G$, Parabolic Subgroups, and Flag Varieties} Let $X := G/K$ be the \emph{symmetric space} of $G$ and fix a basepoint $o = [K] \in X$. Fix a $K$-invariant norm $||\cdot||$ on $\mathfrak{a}$ induced from the Killing form and let $d_{X}$ denote the $G$-invariant metric on $X$ defined by 
\begin{align*}
d_{X}(go,ho) = ||\kappa(g^{-1}h)||
\end{align*} for all $g,h \in G$. We will use the following estimates on the norm of the differences of Cartan projections.
\begin{lem} [Lemma 2.3 of \cite{Ka}]
\label{differencecartanprojections}
For all $g,h \in G$, we have
\begin{align*}
||\kappa(gh) - \kappa(h)|| \leq ||\kappa(g)|| \ \ \mathrm{and} \ \ ||\kappa(gh) - \kappa(g)|| \leq ||\kappa(h)||.
\end{align*}
\end{lem}
We will also need the following result of Benoist, which says that multiplying elements of $G$ by elements in a fixed compact subset of $G$ changes their Cartan projections by at most some bounded error.
\begin{lem} [Lemma 5.1 of \cite{B1}]
\label{Yves'Estimate}
For every compact set $L \subset G$, there exists a compact set $V \subset \mathfrak{a}$ so that for all $g \in G$, we have
\begin{align*}
\kappa(LgL) \subset \kappa(g) + V.
\end{align*}
\end{lem}
The normalizer in $G$ of the nilpotent subalgebra
\begin{align*}
\mathfrak{n} = \bigoplus_{\alpha \in \Sigma^{+}} \mathfrak{g}_{\alpha}
\end{align*} is the \emph{standard minimal parabolic subgroup}, denoted by $P$. The \emph{standard opposite minimal parabolic subgroup} $P^{-} := k_{0}Pk_{0}$ is the normalizer in $G$ of 
\begin{align*}
\mathfrak{n}^{-} = \bigoplus_{\alpha \in \Sigma^{-}} \mathfrak{g}_{\alpha}.
\end{align*} The quotient space $\mathcal{F} := G/P$ is called the \emph{Furstenberg boundary} or \emph{full flag variety} of $G$. We set $\mathcal{F}^{-} := G/P^{-}$ for the \emph{opposite full flag variety}. We can identify $\mathcal{F}^{-}$ with $\mathcal{F}$ via the map $gP^{-} \mapsto gk_{0}P$, although we will not do this in practice so as to avoid any possible confusion. 
\par 
We say that two flags $F_{1} \in \mathcal{F}$ and $F_{2} \in \mathcal{F}^{-}$ are \emph{transverse} if the pair $(F_{1}, F_{2})$ is contained in the open dense $G$-orbit of $(P,P^{-})$ in $\mathcal{F} \times \mathcal{F}^{-}$. In the literature, it is also common to say that the flags $F_{1}$ and $F_{2}$ are in \emph{general position}. For any flag $F \in \mathcal{F}$ (respectively, $F \in \mathcal{F}^{-}$), let $\mathcal{Z}_{F}$ denote the set of flags in $\mathcal{F}^{-}$ (respectively in $\mathcal{F}$) that are $\textbf{not}$ transverse to $F$. Since the $G$-orbit of $(P,P^{-})$ in $\mathcal{F} \times \mathcal{F}^{-}$ is open and dense, the set $\mathcal{Z}_{F}$ is a closed subset with empty interior. Moreover, $\mathcal{Z}_{F} = \mathcal{Z}_{F'}$ if and only if $F = F'$.
\par 
Let $L := P \cap P^{-}$ be the Levi subgroup of $P$, and set $M := K \cap P \subset L$. Then the subgroup $MA$ of $G$ is precisely the stabilizer in $G$ of the point $(P,P^{-}) \in \mathcal{F} \times \mathcal{F}^{-}$.
\par 
We now recall what it means for an element $g \in G$ to be \emph{loxodromic}, and provide an equivalent characterization of this property. A particular type of loxodromic element, which we call an $\epsilon$-\emph{contracting element} and define in Definition \ref{epsiloncontractingelement}, will play a major role in our work.
\begin{defn}
An element $g$ of $G$ is said to be \emph{loxodromic} if $\lambda(g)$ belongs to the interior $\mathfrak{a}^{++}$ of the positive Weyl chamber $\mathfrak{a}^{+}$.
\end{defn}
The following lemma provides a characterization of loxodromic elements in terms of their actions on the Furstenberg boundary $\mathcal{F}$.
\begin{lem} [Lemma 6.39 in \cite{BQ1}]
\label{loxodromiccharacterization2}
Let $G$ be a connected algebraic semisimple real Lie group. An element $g$ of $G$ is loxodromic if and only if it has an attracting fixed point $x_{g}^{+}$ in the Furstenberg boundary $\mathcal{F}$ of $G$.  
\end{lem}
Any loxodromic element $g \in G$ also has a repelling fixed point on the opposite full flag variety $\mathcal{F}^{-}$, which we will denote by $x_{g}^{-}$.
\subsubsection{Anosov semigroups} In this section, we discuss the notion of Anosov semigroups  introduced by Kassel and Potrie in \cite{KP}. Among other things, the authors wished to extend the notion of Anosov representations of discrete subgroups of Lie groups -- initially introduced by Labourie in \cite{La} and further developed by Guichard--Wienhard in \cite{GW} -- to semigroups. However, it is not clear how to adapt the original definition of Anosov representations in this more general setting. Instead, motivated by the notion of dominated splittings for linear cocycles, Kassel--Potrie came up with the definition detailed below.
\par 
Let $\Lambda$ be a semigroup (which may or may not have an identity element $\mathrm{id}$) with a finite generating subset $S$. That is, any element of $\Lambda$ can be written as a product of elements of $S$. For $\gamma \in \Lambda \smallsetminus \{\mathrm{id}\}$, define the \emph{word length of} $\gamma$ \emph{with repsect to} $S$ to be
\begin{align*}
|\gamma|_{S} := \min \{k \geq 1 : \gamma =s_{1} \cdots s_{k}, \ \mathrm{where} \ s_{i} \in S \ \mathrm{for \ all} \ 1 \leq i \leq k \},
\end{align*} and set $|\mathrm{id}|_{S} := 0$. Note that if $S'$ is another finite generating set of $\Lambda$, then there exists a constant $M \geq 1$ so that
\begin{align}
\label{genset}
M^{-1} |\gamma|_{S'} \leq |\gamma|_{S} \leq M |\gamma|_{S'}
\end{align}
for all $\gamma \in \Lambda$.
\begin{defn}
Let $G$ be a connected semisimple real Lie group. A semigroup homomorphism $\rho : \Lambda \rightarrow G$ is said to be P-\emph{Anosov} if there exist constants $C,c > 0$ so that 
\begin{align}
\label{AnosovSemi}
\alpha (\kappa (\rho (\gamma))) \geq C |\gamma|_{S} - c 
\end{align} for all $\gamma \in \Lambda$ and $\alpha \in \Delta$.
\end{defn} If $\Lambda \subset G$ is a semigroup and satisfies (\ref{AnosovSemi}) with $\rho$ being the inclusion $\Lambda \hookrightarrow G$, then we say that $\Lambda$ is a $P$-\emph{Anosov subsemigroup of} $G$. Note that by (\ref{genset}), this definition is independent of the choice of finite generating set for $\Lambda$.
\subsection{Dynamics on flag varieties}
We recall the following well-known and important result about north-south dynamics on flag varieties. We restrict our attention to the case of the full flag varieties $\mathcal{F} = G/P$ and $\mathcal{F}^{-} = G/P^{-}$, although this result also holds for partial flag varieties. The version we present below is from \cite{CZZ2}, although there are many places in the literature where variants of this result have appeared; see for instance Lemma 2.4 of \cite{KOW2}.
\begin{prop} [Proposition $2.3$ of \cite{CZZ2}]
\label{northsouthflagvars}
Suppose $F^{\pm} \in \mathcal{F}^{\pm}$, $\{g_{n}\}_{n \geq 1}$ is a sequence in $G$, and $g_{n} = k_{n} e^{\kappa(g_{n})} \ell_{n}$ is a $KA^{+}K$ decomposition for each $n \geq 1$. The following are equivalent:
\begin{itemize}
    \item[(1)] $k_{n}P \rightarrow F^{+}$, $\ell_{n}^{-1}P^{-} \rightarrow F^{-}$, and $\alpha(\kappa(g_{n})) \rightarrow \infty$ for all $\alpha \in \Delta$.
    \item[(2)] $g_{n}(F) \rightarrow F^{+}$ for all $F \in \mathcal{F} \smallsetminus \mathcal{Z}_{F^{-}}$, and this convergence is uniform on compact subsets of $\mathcal{F} \smallsetminus \mathcal{Z}_{F^{-}}$.
    \item[(3)] $g_{n}^{-1}(F) \rightarrow F^{-}$ for all $F \in \mathcal{F}^{-} \smallsetminus \mathcal{Z}_{F^{+}}$, and this convergence is uniform on compact subsets of $\mathcal{F}^{-} \smallsetminus \mathcal{Z}_{F^{+}}$. 
    \item[(4)] There are open sets $\mathcal{U}^{\pm} \subset \mathcal{F}^{\pm}$ such that $g_{n}(F) \rightarrow F^{+}$ for all $F \in \mathcal{U}^{+}$ and $g_{n}^{-1}(F) \rightarrow F^{-}$ for all $F \in \mathcal{U}^{-}$. 
\end{itemize}
Moreover, when the above holds, for any $\epsilon > 0$ and any compact subsets $K_{1} \subset \mathcal{F} \smallsetminus \mathcal{Z}_{F^{-}}$ and $K_{2} \subset \mathcal{F} \smallsetminus \mathcal{Z}_{F^{+}}$, we have that $g_{n}|_{K_{1}}$ and $g_{n}^{-1}|_{K_{2}}$ are $\epsilon$-Lipschitz for all $n$ sufficiently large. 
\end{prop} The ``moreover'' part is not explicitly stated in Proposition $2.3$ of \cite{CZZ2}, but follows from the proof of this proposition provided in Appendix A of \cite{CZZ2}. The following proposition will be very useful in helping us find $\epsilon$-contracting elements later on in the paper (see Proposition \ref{limitsetcrosslimitset3}).
\begin{prop}
\label{attractingrepellinglocation}
Suppose $F^{+} \in \mathcal{F}$ and $F^{-} \in \mathcal{F}^{-}$ are transverse, and $\{g_{n}\}_{n \geq 1}$ is a sequence of elements of $G$ with $KA^{+}K$ decompositions $g_{n} = k_{n}e^{\kappa(g_{n})}\ell_{n}$ for each $n \geq 1$. If
\begin{align*}
k_{n}P \rightarrow F^{+}, \ \ \ell_{n}^{-1}P^{-} \rightarrow F^{-}, \ \ \mathrm{and} \ \ \alpha(\kappa(g_{n})) \rightarrow \infty \ \ \mathrm{for \ all} \ \ \alpha \in \Delta,
\end{align*} then $g_{n}$ is loxodromic for all $n$ sufficiently large, $x_{g_{n}}^{+} \rightarrow F^{+}$, and $x_{g_{n}}^{-} \rightarrow F^{-}$.
\end{prop}
\begin{proof}
Set 
\begin{align*}
d(F^{+}, \mathcal{Z}_{F^{-}}) := \inf \{d(F^{+},z) : z \in \mathcal{Z}_{F^{-}} \} > 0,
\end{align*} and likewise define $d(F^{-}, \mathcal{Z}_{F^{+}})$, which, since $F^{+}$ and $F^{-}$ are transverse, is also positive. Let
\begin{align*}
0 < \epsilon < \frac{1}{4} \min \Big \{d(F^{+}, \mathcal{Z}_{F^{-}}), d(F^{-}, \mathcal{Z}_{F^{+}}) \Big \}
\end{align*} be arbitrary. Then
\begin{align*}
\overline{B_{\epsilon}(F^{+})} \subset \mathcal{F} \smallsetminus N_{\epsilon}(\mathcal{Z}_{F^{-}}) \ \ \ \mathrm{and} \ \ \ \overline{B_{\epsilon}(F^{-})} \subset \mathcal{F}^{-} \smallsetminus N_{\epsilon}(\mathcal{Z}_{F^{+}}).
\end{align*} By Proposition \ref{northsouthflagvars}, there exists an integer $N \geq 1$ so that for all $n \geq N$, we have 
\begin{align*}
g_{n} \big( \mathcal{F} \smallsetminus N_{\epsilon}(\mathcal{Z}_{F^{-}}) \big) &\subset \overline{B_{\epsilon}(F^{+})}, \\
g_{n}^{-1} \big( \mathcal{F}^{-} \smallsetminus N_{\epsilon}(\mathcal{Z}_{F^{+}}) \big) &\subset \overline{B_{\epsilon}(F^{-})},
\end{align*} and moreover the restrictions 
\begin{align*}
g_{n}|_{\mathcal{F} \smallsetminus N_{\epsilon}(\mathcal{Z}_{F^{-}})} \ \ \ \mathrm{and} \ \ \ g_{n}^{-1}|_{ \mathcal{F}^{-} \smallsetminus N_{\epsilon}(\mathcal{Z}_{F^{+}})}
\end{align*}
are both $\epsilon$-Lipschitz. In particular, for all $n \geq N$, the element $g_{n}$ has an attracting fixed point in $\overline{B_{\epsilon}(F^{+})}$ and a repelling fixed point in $\overline{B_{\epsilon}(F^{-})}$ (namely, the attracting fixed point for $g_{n}^{-1}$). In other words, for $n \geq N$, we see that $g_{n}$ is loxodromic, $d(x_{g_{n}}^{+}, F^{+}) \leq \epsilon$, and $d(x_{g_{n}}^{-}, F^{-}) \leq \epsilon$. As this holds for all $\epsilon > 0$ sufficiently small, this concludes the proof. 
\end{proof}
\section{Background on Discrete Subgroups of Semisimple Lie Groups}
\label{backgroundondiscretesubgroups}
Let $\Gamma < G$ be a Zariski dense discrete subgroup of $G$. In this section, we recall certain objects which play crucial roles in understanding asymptotic properties of the discrete subgroup $\Gamma$.
\subsection{Benoist's limit cone and the limit set}
It is well known that the set $\Gamma_{\mathrm{lox}}$ of loxodromic elements of $\Gamma$ is still Zariski dense in $G$; see for instance Theorem 6.36 of \cite{BQ1} for a proof. In \cite{B2}, Benoist studied certain asymptotic properties of $\Gamma$ by analyzing the image of $\Gamma_{\mathrm{lox}}$ under the Jordan projection $\lambda : G \rightarrow \mathfrak{a}^{+}$. In particular, he introduced a fundamental object in this regard, known as the \emph{Benoist limit cone} (or \emph{limit cone} for short). 
\begin{defn} [Benoist's Limit Cone]
\label{limitcone}
The \emph{limit cone} of $\Gamma$ is the smallest closed cone $\mathcal{L}_{\Gamma}$ in $\mathfrak{a}^{+}$ containing $\lambda(\Gamma_{\mathrm{lox}})$. In other words, $\mathcal{L}_{\Gamma}$ is the closure of the union of the half-lines spanned by the Jordan projections of the loxodromic elements of $\Gamma$:
\begin{align*}
\mathcal{L}_{\Gamma} := \overline{\bigcup_{g \in \Gamma_{\mathrm{lox}}} \mathbb{R}^{+} \lambda(g)}.
\end{align*}
\end{defn}
We remark that the word \emph{cone} does not presume that $\mathcal{L}_{\Gamma}$ is convex, nor that it has non-empty interior. That these properties do in fact hold is a deep result of Benoist:
\begin{thm} [Benoist, Theorem 1.2 of \cite{B2}]
\label{convexnonemptyinterior}
Let $G$ be a connected algebraic semisimple real Lie group and let $\Gamma$ be a discrete subgroup of $G$. Then the limit cone $\mathcal{L}_{\Gamma}$ is a convex cone with non-empty interior.
\end{thm}
The limit cone also has the following important properties:
\begin{itemize}
    \item[(a)] The limit cone $\mathcal{L}_{\Gamma}$ contains $\lambda(\Gamma)$. 
    \item[(b)] The limit cone $\mathcal{L}_{\Gamma}$ is the asymptotic cone of the image of $\Gamma$ under the Cartan projection, that is, 
    \begin{align*}
    \mathcal{L}_{\Gamma} = \{v \in \mathfrak{a}^{+} \ : \ \exists\{g_{n}\} \subset \Gamma, \ \exists \{t_{n}\} \subset \mathbb{R} \ \mathrm{with} \ t_{n} \searrow 0 \ \mathrm{so \ that} \ \lim_{n \to \infty} t_{n} \kappa(g_{n}) = v \}.
    \end{align*}
\end{itemize}
We refer the reader to Benoist's original work \cite{B2} for proofs of these results. 
\begin{rem}
All of the above definitions and results remain unchanged if $\Gamma$ is only assumed to be a Zariski dense discrete \emph{subsemigroup} of $G$, and not necessarily a subgroup of $G$. See \cite{B2} for details.
\end{rem}
In the same paper, Benoist introduced and studied the \emph{limit set} $\Lambda(\Gamma) \subset \mathcal{F}$ of a Zariski dense discrete subgroup $\Gamma < G$. It is a classical object when $G$ has real rank one and for $\Gamma < \mathrm{SL}(n, \mathbb{R})$, $n \geq 3$, it was introduced and studied earlier by Guivarc'h in \cite{Gu}. Among other things, Benoist showed that it is the unique non-empty, perfect, $\Gamma$-invariant closed subset of $\mathcal{F}$ on which $\Gamma$ acts minimally. Let $\nu$ denote the $K$-invariant probability measure on $\mathcal{F} = G/P$. Formally, the limit set is defined as follows:
\begin{defn} [Limit point and limit set] 
\label{limitset} ~\
\begin{itemize} 
    \item[(a)] A point $x \in \mathcal{F}$ for which there exists a sequence $\{g_{n}\}_{n \geq 1}$ of elements in $G$ so that $(g_{n})_{*}\nu \overset{*}{\rightharpoonup} \delta_{x}$ (the Dirac mass at $x$) is called a \emph{limit point} of this sequence.
    \item[(b)] A \emph{limit point of} $\Gamma$ \emph{in} $\mathcal{F}$ is any point as in item (a), with the additional constraint that the sequence $\{g_{n}\}_{n \geq 1}$ is a sequence of elements in $\Gamma$.
    \item[(c)] The \emph{limit set of} $\Gamma$ \emph{in} $\mathcal{F}$, denoted $\Lambda(\Gamma)$, is the set of all limit points of $\Gamma$ in $\mathcal{F}$.
\end{itemize}
\end{defn} Recall that if $g \in G$ satisfies $\alpha(\kappa(g)) > 0$ for all $\alpha \in \Delta$, then the flag $k_{g}P \in \mathcal{F}$ is independent of the choice of $KA^{+}K$ decomposition $g = k_{g}e^{\kappa(g))} \ell_{g}$ of the element $g$, hence is well-defined. Using Lemma 3.5 in \cite{B2}, one can show that the limit set of $\Gamma$ coincides with the set of accumulation points of sequences of the form $\{k_{\gamma_{n}}P\}_{n \geq 1}$ where $\{\gamma_{n}\}_{n \geq 1} \subset \Gamma$ is such that $\min_{\alpha \in \Delta} \alpha (\kappa (\gamma_{n})) \rightarrow \infty$. Lastly, we remark that we may similarly define \emph{the limit set of} $\Gamma$ \emph{in} $\mathcal{F}^{-}$, which we denote by $\Lambda(\Gamma)^{-}$. All of the properties and characterizations of $\Lambda(\Gamma)$ hold as well for $\Lambda(\Gamma)^{-}$.
\subsection{Quint's growth indicator function} 
Given an open cone $\mathcal{C} \subset \mathfrak{a}^{+}$, set
\begin{align*}
\Gamma_{\mathcal{C}} = \{\gamma \in \Gamma : \kappa(\gamma) \in \mathcal{C} \}.
\end{align*}
In \cite{Q1}, Quint introduced his \emph{growth indicator function}, which he later used in \cite{Q2} to study Patterson--Sullivan measures for Zariski dense discrete subgroups of higher-rank Lie groups. It is a higher-rank analog of the critical exponent; precisely, it is the function $\psi_{\Gamma} : \mathfrak{a}^{+} \rightarrow \mathbb{R} \cup \{- \infty\}$ defined by
\begin{align*}
\psi_{\Gamma}(v) = ||v|| \cdot \inf_{\mathcal{C} \ni v} \tau_{\mathcal{C}},
\end{align*} where the infimum is taken over all open cones $\mathcal{C} \subset \mathfrak{a}^{+}$ containing $v$ and where 
\begin{align*}
\tau_{C} := \limsup_{T \to \infty} \frac{1}{T} \log \# \{\gamma \in \Gamma : \kappa(\gamma) \in \mathcal{C}, \ ||\kappa(\gamma)|| \leq T \}.
\end{align*} Equivalently, $\tau_{C}$ is the abscissa of convergence of the Poincar\'e series 
\begin{align*}
    Q_{\Gamma_{\mathcal{C}}}(s) := \sum_{\gamma \in \Gamma_{\mathcal{C}}} e^{-s ||\kappa(\gamma)||},
\end{align*} that is, 
\begin{align*}
\tau_{\mathcal{C}} = \inf \{s > 0 : Q_{\Gamma_{\mathcal{C}}}(s) < \infty \}.
\end{align*} We remark that these definitions are independent of the choice of norm $||\cdot||$ on $\mathfrak{a}^{+}$, hence the growth indicator is well-defined. Quint showed in \cite{Q1}
that $\psi_{\Gamma}$ is a concave upper semi-continuous function satisfying 
\begin{align*}
\mathcal{L}_{\Gamma} = \{v \in \mathfrak{a}^{+} : \psi_{\Gamma}(v) \geq 0 \},
\end{align*}
and moreover $\psi_{\Gamma} > 0$ on $\mathrm{int}(\mathcal{L}_{\Gamma})$. Note also that $\psi_{\Gamma} \equiv - \infty$ outside of $\mathcal{L}_{\Gamma}$. Recall that the \emph{critical exponent} of $\Gamma$ is given by 
\begin{align*}
\delta(\Gamma) &:= \inf \Big\{s > 0 : Q_{\Gamma}(s) := \sum_{\gamma \in \Gamma} e^{-s||\kappa(\gamma)||} < \infty \Big\} \\
&= \limsup_{T \to \infty} \frac{1}{T} \log \# \{\gamma \in \Gamma : ||\kappa(\gamma)|| \leq T \}.
\end{align*}
From the definitions of the growth indicator and critical exponent, it is immediate that, for any unit vector $v \in \mathfrak{a}^{+}$, $\psi_{\Gamma}(v) \leq \delta(\Gamma)$. The following result, due to Quint \cite{Q1}, shows that there is at least one direction in the positive Weyl chamber where the growth indicator attains this upper bound (in fact, there is a unique such direction, but we will not need this). We provide a proof since it is short and may be helpful for readers less familiar with the growth indicator function. In what follows, we let $\mathbb{S} := \{w \in \mathfrak{a} : ||w|| = 1 \}$.
\begin{prop}
\label{MaximizingGrowthInd}
There exists a unit vector $v \in \mathfrak{a}^{+}$ so that $\psi_{\Gamma}(v) = \delta(\Gamma)$.
\end{prop}
\begin{proof}
If not, then $\psi_{\Gamma}(v) < \delta(\Gamma)$ for all $v \in \mathbb{S} \cap \mathfrak{a}^{+}$. This means that for every $v \in \mathbb{S} \cap \mathfrak{a}^{+}$, there exists $0 < \delta_{v} < \delta(\Gamma)$ and an open cone $\mathcal{C}_{v} \subset \mathfrak{a}^{+}$ containing $v$ such that 
\begin{align*}
    Q_{\Gamma_{\mathcal{C}_{v}}}(\delta_{v}) = \sum_{\gamma \in \Gamma_{\mathcal{C}_{v}}} e^{-\delta_{v} ||\kappa(\gamma)||} < \infty.
\end{align*} Since $\mathbb{S} \cap \mathfrak{a}^{+}$ is compact, there exist finitely many $v_{1}, \dots, v_{k} \in \mathbb{S} \cap \mathfrak{a}^{+}$ so that 
\begin{align*}
    \mathbb{S} \cap \mathfrak{a}^{+} \subset \mathbb{S} \cap \bigcup_{i=1}^{k} \mathcal{C}_{v_{i}},
\end{align*} hence $\mathfrak{a}^{+} \subset \bigcup_{i=1}^{k} \mathcal{C}_{v_{i}}$. But then setting $\delta := \max_{1 \leq i \leq k} \delta_{v_{i}} < \delta(\Gamma)$, we obtain 
\begin{align*}
    Q_{\Gamma}(\delta) = \sum_{\gamma \in \Gamma} e^{-\delta ||\kappa(\gamma)||} \leq \sum_{i=1}^{k} Q_{\Gamma_{\mathcal{C}_{v}}} (\delta_{v}) < \infty,
\end{align*} which is impossible. This concludes the proof.
\end{proof}
\begin{cor}
\label{LargeGrowthInd}
For any $0 < \delta < \delta(\Gamma)$, there exists a unit vector $u \in \mathfrak{a}^{++}$ so that $\psi_{\Gamma}(u) > \delta$.
\end{cor}
\begin{proof}
Let $0 < \delta < \delta(\Gamma)$ be arbitrary. By Proposition \ref{MaximizingGrowthInd}, there exists a unit vector $v \in \mathfrak{a}^{+}$ so that $\psi_{\Gamma}(v) = \delta(\Gamma)$. If $v \in \mathfrak{a}^{++}$, then we are done. So suppose that $v \in \mathfrak{a}^{+} \smallsetminus \mathfrak{a}^{++}$. Since $\Gamma$ is assumed to be Zariski dense in $G$, the interior of the Benoist limit cone is nonempty, hence we may fix some $w \in \mathrm{int}(\mathcal{L}_{\Gamma})$. Since the growth indicator function is homogeneous and is also strictly positive on $\mathrm{int}(\mathcal{L}_{\Gamma})$, we have 
\begin{align}
\label{scalinggrowthind}
r \cdot \psi_{\Gamma}(w) = \psi_{\Gamma}(rw) > 0,
\end{align}
for all $r>0$. Fix $\frac{\delta}{\delta(\Gamma)} < t < 1$ and now let $s > 0$ be such that the element 
\begin{align*}
u := tv + (1-t)sw
\end{align*}
of $\mathfrak{a}^{+}$ has norm one. Notice that $u \in \mathfrak{a}^{++}$. Indeed, for any $\alpha \in \Delta$, we have
\begin{align*}
\alpha (u) = t \cdot \alpha (v) + (1-t)s \cdot \alpha (w) = (1-t)s \cdot \alpha (w) > 0
\end{align*} as $w \in \mathrm{int}(\mathcal{L}_{\Gamma}) \subset \mathfrak{a}^{++}$. From (\ref{scalinggrowthind}) and the concavity and homogeneity of $\psi_{\Gamma}$, we obtain
\begin{align*}
    \psi_{\Gamma}(u) \geq t \cdot \psi_{\Gamma}(v) + (1-t)s \cdot \psi_{\Gamma}(w) > t \cdot \delta (\Gamma) > \delta,
\end{align*} as desired.
\end{proof}

\section{Uniformly Contracting Loxodromic Elements}
\label{contractingelementssection}
Motivated by the work of Abels--Margulis--Soifer \cite{AMS} and Benoist \cite{B1}, \cite{B2} where the notion of $\epsilon$-proximal elements and related concepts proved very fruitful, we develop properties of loxodromic elements we call $\epsilon$-\emph{contracting elements} of $G$. These elements were first introduced by Guichard--Wienhard (Definition 5.6 of \cite{GW}), but our work appears to be the first to develop basic properties of these elements (analogous to those of $\epsilon$-proximal elements) and utilize them heavily to study arbitrary Zariski dense discrete subgroups of Lie groups. One particularly nice feature of these objects is that they will allow us to study the action of the Lie group $G$ on its Furstenberg boundary $\mathcal{F}$ intrinsically, in the sense that we avoid embedding $\mathcal{F}$ into the product of projective spaces associated to the highest weight Tits representations \cite{Ti}. 
\par 
As before, if $g \in G$ is loxodromic, we let $x_{g}^{+}$ denote its unique attracting fixed point in $\mathcal{F} = G/P$ and $x_{g}^{-}$ its unique repelling fixed point in $\mathcal{F}^{-} = G/P^{-}$. For any subset $Z$ of $\mathcal{F}$ and $\epsilon > 0$, we let 
\begin{align*}
N_{\epsilon}(Z) := \big \{x \in \mathcal{F} : \inf_{z \in Z} d(x,z) < \epsilon \big \}
\end{align*}
denote the open $\epsilon$-neighborhood of $Z$ in $\mathcal{F}$. For a point $x \in \mathcal{F}$, we write $B_{\epsilon}(x)$ for the open ball of radius $\epsilon$ centered at $x$ in $\mathcal{F}$.
\begin{defn} [$\epsilon$-contracting element]
\label{epsiloncontractingelement}
Given $\epsilon > 0$, we say that a loxodromic element $g \in G$ is $\epsilon$-\emph{contracting} if the following three conditions are satisfied:
\begin{itemize}
    \item[(a)] $d \big(x_{g}^{+}, \mathcal{Z}_{x_{g}^{-}} \big) := \inf \big\{d(x_{g}^{+}, z) : z \in \mathcal{Z}_{x_{g}^{-}} \big \} \geq 2 \epsilon$, 
    \item[(b)] $g \big( \mathcal{F} \smallsetminus N_{\epsilon} (\mathcal{Z}_{x_{g}^{-}}) \big) \subset \overline{B_{\epsilon}(x_{g}^{+})}$,
    \item[(c)] $g \big|_{\mathcal{F} \smallsetminus N_{\epsilon} (\mathcal{Z}_{x_{g}^{-}})}$ is $\epsilon$-Lipschitz.
\end{itemize}
\end{defn}
The following lemma is an analog of a result of Benoist (Lemma 6.2 in \cite{B1}). It provides sufficient conditions for an element $g \in G$ to be $\epsilon$-contracting and moreover provides control over the location of its attracting fixed point $x_{g}^{+} \in \mathcal{F}$ and the set of elements $\mathcal{Z}_{x_{g}^{-}} \subset \mathcal{F}$ which are not transverse to its repelling fixed point $x_{g}^{-} \in \mathcal{F}^{-}$. 
\begin{lem}
\label{criterionforcontraction}
Let $g \in G \smallsetminus \{\mathrm{id}\}$, $x^{+} \in \mathcal{F}$, $x^{-} \in \mathcal{F}^{-}$, and $0 < \epsilon < 1$. Suppose that 
\begin{align*}
d(x^{+}, \mathcal{Z}_{x^{-}}) \geq 6 \epsilon, \ \  g \big( \mathcal{F} \smallsetminus N_{\epsilon}(\mathcal{Z}_{x^{-}}) \big) \subset \overline{B_{\epsilon}(x^{+})}, \ \ \mathrm{and} \ \ g \big|_{\mathcal{F} \smallsetminus N_{\epsilon}(\mathcal{Z}_{x^{-}})} \ \mathrm{is \ } \epsilon \mathrm{-Lipschitz}.
\end{align*} Then $g$ is $2 \epsilon$-contracting, $d(x_{g}^{+}, x^{+}) \leq \epsilon$, and $d_{\mathrm{Haus}}(\mathcal{Z}_{x_{g}^{-}}, \mathcal{Z}_{x^{-}}) < \epsilon$.
\end{lem}
\begin{proof}
By assumption, the restriction of $g$ to $\mathcal{F} \smallsetminus N_{\epsilon}(\mathcal{Z}_{x^{-}})$ is an $\epsilon$-Lipschitz map. Since $0 < \epsilon < 1$, it follows that $g$ has an attracting fixed point in $\mathcal{F}$. Hence $g$ is loxodromic by Lemma \ref{loxodromiccharacterization2}. As
\begin{align*}
g \big( \mathcal{F} \smallsetminus N_{\epsilon}(\mathcal{Z}_{x^{-}}) \big) \subset \overline{B_{\epsilon}(x^{+})},
\end{align*}
we have $d(x_{g}^{+}, x^{+}) \leq \epsilon$. Since $g^{n}y \rightarrow x_{g}^{+}$ as $n \to \infty$ for all $y \in \mathcal{F} \smallsetminus N_{\epsilon}(\mathcal{Z}_{x^{-}})$, we have
\begin{align*}
    \mathcal{Z}_{x_{g}^{-}} \cap (\mathcal{F} \smallsetminus N_{\epsilon}(\mathcal{Z}_{x^{-}})) = \emptyset,
\end{align*} that is, $d_{\mathrm{Haus}}(\mathcal{Z}_{x_{g}^{-}}, \mathcal{Z}_{x^{-}}) < \epsilon$. By assumption, $d(x^{+}, \mathcal{Z}_{x^{-}}) \geq 6 \epsilon$. It follows that $d(x_{g}^{+}, \mathcal{Z}_{x_{g}^{-}}) \geq 4 \epsilon$ and moreover
\begin{align*}
g \big( \mathcal{F} \smallsetminus N_{2\epsilon}(\mathcal{Z}_{x_{g}^{-}}) \big) \subset g \big( \mathcal{F} \smallsetminus N_{\epsilon}(\mathcal{Z}_{x^{-}}) \big) \subset \overline{B_{\epsilon}(x^{+})} \subset \overline{B_{2 \epsilon}(x_{g}^{+})}.
\end{align*} Lastly, $g|_{\mathcal{F} \smallsetminus N_{2 \epsilon}(\mathcal{Z}_{x_{g}^{-}})}$ is $\epsilon$-Lipschitz, hence trivially $2 \epsilon$-Lipschitz. This shows that $g$ is $2 \epsilon$-contracting, as desired. 
\end{proof}
The following proposition, which is motivated by Proposition 6.4 in Benoist's paper \cite{B1} (and, more broadly, by Benoist's notion of $\epsilon$-Schottky semigroups and subgroups used very fruitfully in his papers \cite{B1} and \cite{B2}), will be very important in our construction of asymptotically large free subsemigroups of $\Gamma$.
\begin{prop}
\label{contractingsemigroup}
Fix $\epsilon > 0$ and suppose that $g_{1}, \dots, g_{l}$ are $\epsilon$-contracting elements of $G$ such that, for all $1 \leq i \neq j \leq l$, we have
\begin{align*}
    d(x_{g_{i}}^{+}, \mathcal{Z}_{x_{g_{j}}^{-}}) \geq 6 \epsilon.
\end{align*} Then every element of the semigroup generated by the set $S := \{g_{1}, \dots, g_{l}\}$ is either $\epsilon$-contracting or $2 \epsilon$-contracting. Moreover, if $g = h_{1} \cdots h_{k}$ where $h_{i} \in S$ for all $1 \leq i \leq k$, then 
\begin{align*}
    d(x_{g}^{+}, x_{h_{1}}^{+}) \leq \epsilon \ \ \mathrm{and} \ \ d_{\mathrm{Haus}}\big(\mathcal{Z}_{x_{g}^{-}}, \mathcal{Z}_{x_{h_{k}}^{-}} \big) < \epsilon.
\end{align*}
\end{prop}
\begin{proof}
Notice that if $h \in G$ is $\epsilon$-contracting, then so is $h^{n}$ for any $n \geq 1$. Thus it suffices to show that if $g := h_{1} \cdots h_{k}$ where $h_{1}, \dots, h_{k} \in S$ and $h_{i} \neq h_{i+1}$ for all $1 \leq i \leq k-1$, then $g$ is either $\epsilon$-contracting or $2\epsilon$-contracting. By hypothesis, each $h_{i}$ is $\epsilon$-contracting and moreover
\begin{align*}
    d(x_{h_{i+1}}^{+}, \mathcal{Z}_{x_{h_{i}}^{-}}) \geq 6 \epsilon
\end{align*} for all $1 \leq i \leq k-1$. Hence
\begin{align*}
    h_{k} \Big(\mathcal{F} \smallsetminus N_{\epsilon}\Big(\mathcal{Z}_{x_{h_{k}}^{-}}\Big) \Big) \subset \overline{B_{\epsilon}(x_{h_{k}}^{+})} \subset \mathcal{F} \smallsetminus N_{\epsilon}\Big(\mathcal{Z}_{x_{h_{k-1}}^{-}} \Big).
\end{align*} By induction, we see that 
\begin{align*}
    h_{2} \cdots h_{k} \Big(\mathcal{F} \smallsetminus N_{\epsilon}\Big(\mathcal{Z}_{x_{h_{k}}^{-}}\Big) \Big) \subset \mathcal{F} \smallsetminus N_{\epsilon}\Big(\mathcal{Z}_{x_{h_{1}}^{-}}\Big).
\end{align*} It follows that
\begin{align}
\label{productnorthsouth}
g \Big(\mathcal{F} \smallsetminus N_{\epsilon}\Big(\mathcal{Z}_{x_{h_{k}}^{-}}\Big) \Big) = h_{1} h_{2} \cdots h_{k} \Big(\mathcal{F} \smallsetminus N_{\epsilon}\Big(\mathcal{Z}_{x_{h_{k}}^{-}}\Big) \Big) \subset \overline{B_{\epsilon}(x_{h_{1}}^{+})}.
\end{align} There are now two cases to consider.
\par 
$\textbf{Case 1}:$ Suppose first that $h_{1} \neq h_{k}$. Then 
\begin{align*}
d(x_{h_{1}}^{+}, \mathcal{Z}_{x_{h_{k}}^{-}}) \geq 6 \epsilon, \ \ g\big|_{\mathcal{F} \smallsetminus N_{\epsilon}\big(\mathcal{Z}_{x_{h_{k}}^{-}} \big)} \ \ \mathrm{is} \ \ \epsilon-\mathrm{Lipschitz}, \ \ \mathrm{and} \ (\ref{productnorthsouth}) \ \mathrm{holds.}
\end{align*}
By Lemma \ref{criterionforcontraction}, we conclude that $g$ is a $2 \epsilon$-contracting element with 
\begin{align*}
d(x_{g}^{+}, x_{h_{1}}^{+}) \leq \epsilon \ \ \ \mathrm{and} \ \ \ d_{\mathrm{Haus}}\big(\mathcal{Z}_{x_{g}^{-}}, \mathcal{Z}_{x_{h_{k}}^{-}} \big) < \epsilon.
\end{align*}
$\textbf{Case 2}:$ Suppose instead that $h_{1} = h_{k}$. Then we can rewrite (\ref{productnorthsouth}) as 
\begin{align*}
g \Big(\mathcal{F} \smallsetminus N_{\epsilon}\Big(\mathcal{Z}_{x_{h_{1}}^{-}}\Big) \Big) = h_{1} h_{2} \cdots h_{k-1}h_{1} \Big(\mathcal{F} \smallsetminus N_{\epsilon}\Big(\mathcal{Z}_{x_{h_{1}}^{-}}\Big) \Big) \subset \overline{B_{\epsilon}(x_{h_{1}}^{+})},
\end{align*} from which it follows that $g = h_{1} h_{2} \cdots h_{k-1} h_{1}$ is $\epsilon$-contracting with
\begin{align*}
d(x_{g}^{+}, x_{h_{1}}^{+}) \leq \epsilon \ \ \ \mathrm{and} \ \ \ d_{\mathrm{Haus}}\big(\mathcal{Z}_{x_{g}^{-}}, \mathcal{Z}_{x_{h_{1}}^{-}} \big) < \epsilon.
\end{align*}
This concludes the proof. 
\end{proof}
Inspired by the work of Benoist \cite{B1}, \cite{B2} and Quint \cite{Q1}, \cite{Q2}, we consider the ``shadows'' of $\epsilon$-contracting elements on the Furstenberg boundary $\mathcal{F}$ of $G$ (see also \cite{BCZZ1} for a somewhat similar notion of shadows in the context of a convergence group acting on a compact metrizable space). We emphasize that, in contrast with other notions of shadows present in the literature, these shadows are \emph{only defined for $\epsilon$-contracting elements}.
\begin{defn} [Shadow]
\label{contractingshadow}
Let $g \in G$ be an $\epsilon$-contracting element. For $r > 0$, the $r$-\emph{shadow} of $g$ is defined to be
\begin{align*}
\mathcal{S}_{r}(g) = g \big( \mathcal{F} \smallsetminus N_{r} (\mathcal{Z}_{x_{g}^{-}}) \big).
\end{align*}
\end{defn} Notice that, by definition, we have $\mathcal{S}_{r}(g) \subset \overline{B_{\epsilon}(x_{g}^{+})}$ for all $r \geq \epsilon$. 
\begin{prop} 
\label{shadowinclusion}
With the same hypotheses and notation as in Proposition \ref{contractingsemigroup}, let $m \geq 1$ and $\gamma \in S^{m}$ be arbitrary. If $\eta = \gamma \zeta$ for some $\zeta \in S$, then 
\begin{align*}
\mathcal{S}_{2 \epsilon} (\eta) \subset \mathcal{S}_{4 \epsilon}(\gamma).
\end{align*}
\end{prop}
\begin{proof}
Since $\eta = \gamma \zeta$, to show that $\mathcal{S}_{2 \epsilon}(\eta) \subset \mathcal{S}_{4 \epsilon}(\gamma)$, it is equivalent to prove that 
\begin{align}
\label{inclusion1}
\zeta (\mathcal{F} \smallsetminus N_{2 \epsilon}(\mathcal{Z}_{x_{\eta}^{-}})) \subset \mathcal{F} \smallsetminus N_{4 \epsilon}(\mathcal{Z}_{x_{\gamma}^{-}}).
\end{align} Write $\gamma = h_{1} \cdots h_{m}$, where $h_{i} \in S$ for all $1 \leq i \leq m$. By assumption, we have $d(x_{\zeta}^{+}, \mathcal{Z}_{x_{h_{m}}^{-}}) \geq 6 \epsilon$, hence
\begin{align}
\label{inclusion2}
\overline{B_{\epsilon}(x_{\zeta}^{+})} \subset \mathcal{F} \smallsetminus N_{5 \epsilon}(\mathcal{Z}_{x_{h_{m}}^{-}}).
\end{align} By Proposition \ref{contractingsemigroup}, we also have
\begin{align}
\label{inclusion3}
d_{\mathrm{Haus}}(\mathcal{Z}_{x_{\gamma}^{-}}, \mathcal{Z}_{x_{h_{m}}^{-}}) < \epsilon.
\end{align} Combining (\ref{inclusion2}) and (\ref{inclusion3}) gives
\begin{align}
\label{inclusion4}
\overline{B_{\epsilon}(x_{\zeta}^{+})} \subset \mathcal{F} \smallsetminus N_{5 \epsilon}(\mathcal{Z}_{x_{h_{m}}^{-}}) \subset \mathcal{F} \smallsetminus N_{4 \epsilon}(\mathcal{Z}_{x_{\gamma}^{-}}).
\end{align} Since $\eta = \gamma \zeta$ with $\zeta \in S$, Proposition \ref{contractingsemigroup} gives $d_{\mathrm{Haus}}(\mathcal{Z}_{x_{\eta}^{-}}, \mathcal{Z}_{x_{\zeta}^{-}}) < \epsilon$, hence
\begin{align}
\label{inclusion5}
\mathcal{F} \smallsetminus N_{2 \epsilon}(\mathcal{Z}_{x_{\eta}^{-}}) \subset \mathcal{F} \smallsetminus  N_{\epsilon}(\mathcal{Z}_{x_{\zeta}^{-}}).
\end{align} Using (\ref{inclusion5}), (\ref{inclusion4}), and the fact that $\zeta \in S$ is $\epsilon$-contracting, we conclude that 
\begin{align*}
\zeta (\mathcal{F} \smallsetminus N_{2 \epsilon}(\mathcal{Z}_{x_{\eta}^{-}})) \subset \zeta (\mathcal{F} \smallsetminus  N_{\epsilon}(\mathcal{Z}_{x_{\zeta}^{-}})) \subset \overline{B_{\epsilon}(x_{\zeta}^{+})} \subset \mathcal{F} \smallsetminus N_{4 \epsilon}(\mathcal{Z}_{x_{\gamma}^{-}}). 
\end{align*} This verifies (\ref{inclusion1}), thereby completing the proof of the proposition.
\end{proof}
\section{Certain Asymptotic Phenomena of Discrete Subgroups}
\label{asymptoticphenomena}
Let $\Gamma < G$ be a Zariski dense discrete subgroup. For $\gamma \in \Gamma$, let 
\begin{align*}
\gamma = k_{\gamma} e^{\kappa(\gamma)} \ell_{\gamma} \in KA^{+}K
\end{align*}
be a Cartan decomposition.
Fix $0 < \delta < \delta + r < \delta(\Gamma)$. By Corollary \ref{LargeGrowthInd}, there exists a unit vector $u \in \mathfrak{a}^{++}$ so that
\begin{align*}
    \psi_{\Gamma}(u) = \inf_{\mathcal{C} \ni u} \tau_{\mathcal{C}} > \delta + r.
\end{align*}
By definition, we have
\begin{align}
\label{given1}
Q_{\Gamma_{\mathcal{C}}}(\delta + r) = \sum_{\gamma \in \Gamma_{\mathcal{C}}} e^{-(\delta + r) ||\kappa(\gamma)||} = \infty
\end{align} for \emph{every open cone} $\mathcal{C} \subset \mathfrak{a}^{++}$ containing $u$. We emphasize that the fact that this Poincar\'e series diverges for \emph{every} such open cone will be a crucial ingredient in our construction of asymptotically large free subsemigroups of $\Gamma$. Now given an open cone $\mathcal{C} \subset \mathfrak{a}^{++}$ containing $u \in \mathfrak{a}^{++}$, points $x \in \mathcal{F}$, $y \in \mathcal{F}^{-}$, and constants $\epsilon > 0$ and $n \geq 1$, define
\begin{align*}
\Gamma_{\mathcal{C}, x,y,n, \epsilon} := \{ \gamma \in \Gamma_{C} \ : \ ||\kappa(\gamma)|| \geq n, \ \ d(k_{\gamma}P, x)  < \epsilon, \ \ \mathrm{and} \ \ d_{\mathrm{Haus}}(\mathcal{Z}_{\ell_{\gamma}^{-1}P^{-}}, \mathcal{Z}_{y}) < \epsilon \}.
\end{align*}
\begin{obs}
\label{eventuallyloxodromic}
For every $\epsilon > 0$, there exists an integer $n_{0} = n_{0}(\epsilon) \geq 1$ so that for all $n \geq n_{0}$, every element of $\Gamma_{\mathcal{C}, x,y,n, \epsilon}$ is loxodromic.
\end{obs}
\begin{proof}
If not, then there exists $\epsilon > 0$, a sequence $\{m_{n}\}$ of integers with $m_{n} \rightarrow \infty$, and group elements $\gamma_{n} \in \Gamma_{\mathcal{C}, x,y, m_{n}, \epsilon}$ none of which are loxodromic. Passing to a subsequence, we can assume without loss of generality that 
\begin{align*}
k_{\gamma_{m_{n}}}P \rightarrow F^{+} \in \mathcal{F} \ \ \ \mathrm{and} \ \ \  \ell_{\gamma_{m_{n}}}^{-1}P^{-} \rightarrow F^{-} \in \mathcal{F}^{-}.
\end{align*}
Notice also that
\begin{align*}
\min_{\alpha \in \Delta} \alpha(\kappa(\gamma_{m_{n}})) \rightarrow \infty,
\end{align*}
since each $\kappa(\gamma_{m_{n}})$ is contained in the open cone $\mathcal{C} \subset \mathfrak{a}^{++}$ and $||\kappa(\gamma_{m_{n}})|| \geq m_{n}$. Then the ``moreover'' part of Proposition \ref{northsouthflagvars} implies that, for $n$ sufficiently large, $\gamma_{m_{n}}$ has an attracting fixed point in $\mathcal{F}$, and therefore $\gamma_{m_{n}}$ is loxodromic. This is a contradiction, which concludes the proof.
\end{proof}
Now let $\mathcal{C}'$ be an open cone so that
\begin{align*}
\overline{\mathcal{C}'} \subsetneq \mathcal{C} \subset \mathfrak{a}^{++}.
\end{align*}
Recall that $G$ acts on $\mathcal{F}$ by Lipschitz transformations; see for instance section 5 of \cite{Q2}. For each $g \in G$, let $L_{g}$ denote the Lipschitz constant for the action of $g$ on $\mathcal{F}$. 
\begin{lem}
\label{inclusionofomegas}
For every $g \in \Gamma$, $\epsilon > 0$, and $(x,y) \in \Lambda(\Gamma) \times \Lambda(\Gamma)^{-}$, there exist $N = N(x,y,g, \epsilon) \geq 1$ and $C = C(x,y,g,\epsilon) > 0$ so that 
\begin{align*}
g \cdot \Gamma_{\mathcal{C}', x, y, n+C, \frac{\epsilon}{2L_{g}}} \subset \Gamma_{\mathcal{C}, gx, y, n, \epsilon},
\end{align*}
for all $n \geq N$. 
\end{lem}
\begin{proof}
Suppose not. Then there exist $g \in \Gamma$, $\epsilon > 0$, a pair $(x,y) \in \Lambda(\Gamma) \times \Lambda(\Gamma)^{-}$, and sequences $\{k_{m}\} \subset \mathbb{N}$ and $\{\gamma_{m}\} \subset \Gamma$ such that 
\begin{align*}
\gamma_{m} \in \Gamma_{\mathcal{C}', x, y, k_{m}+m, \frac{\epsilon}{2L_{g}}} \ \ \ \mathrm{but} \ \ \ g \gamma_{m} \notin \Gamma_{\mathcal{C}, gx, y, k_{m}, \epsilon}, 
\end{align*}
for all $m \geq 1$. Pass to a subsequence $\{\gamma_{m_{j}}\} \subset \{\gamma_{m}\}$ so that 
\begin{align*}
\min_{\alpha \in \Delta} \alpha(\kappa(\gamma_{m_{j}})) \rightarrow \infty, \ \ k_{\gamma_{m_{j}}}P \rightarrow F^{+} \in \mathcal{F} \ \ \mathrm{and} \ \  \ell_{\gamma_{m_{j}}}^{-1}P^{-} \rightarrow F^{-} \in \mathcal{F}^{-}.
\end{align*}

By Lemma \ref{differencecartanprojections}, we know that
\begin{align*}
||\kappa(g\gamma_{m_{j}}) - \kappa(\gamma_{m_{j}})|| \leq ||\kappa(g)||,
\end{align*}
for all $j \geq 1$. Since $\overline{\mathcal{C}'} \subsetneq \mathcal{C}$, we have
\begin{align}
\label{omegafact1}
\kappa(g \gamma_{m_{j}}) \in \mathcal{C}
\end{align}
for all $j$ sufficiently large. Moreover,
\begin{align}
\label{omegafact2}
\notag ||\kappa(g \gamma_{m_{j}})|| &\geq ||\kappa(g^{-1} g \gamma_{m_{j}})|| - ||\kappa(g^{-1})|| \\
\notag &= ||\kappa(\gamma_{m_{j}})|| - ||\kappa(g^{-1})|| \\
\notag &\geq k_{m_{j}} + m_{j} - ||\kappa(g^{-1})|| \\
&\geq k_{m_{j}}
\end{align} as soon as $m_{j} \geq ||\kappa(g^{-1})||$. By Proposition \ref{northsouthflagvars}, we have 
\begin{align*}
k_{g\gamma_{m_{j}}}P \rightarrow gF^{+} \ \ \ \mathrm{and} \ \ \ \ell_{g \gamma_{m_{j}}}^{-1} P^{-} \rightarrow F^{-}.
\end{align*}
Thus for all $j$ sufficiently large, both (\ref{omegafact1}) and (\ref{omegafact2}) hold, and also
\begin{align*}
d(k_{g \gamma_{m_{j}}} P, gF^{+}) < \frac{\epsilon}{2} \ \ \mathrm{and} \ \ d_{\mathrm{Haus}}(\mathcal{Z}_{\ell_{g \gamma_{m_{j}}}^{-1} P^{-}}, \mathcal{Z}_{F^{-}}) < \frac{\epsilon}{2}.
\end{align*} By assumption, we have 
\begin{align*}
d(F^{+}, x) < \frac{\epsilon}{2L_{g}} \ \ \mathrm{and} \ \ d_{\mathrm{Haus}}(\mathcal{Z}_{F^{-}}, \mathcal{Z}_{y}) < \frac{\epsilon}{2L_{g}}.
\end{align*}
But since $g$ acts on $\mathcal{F}$ by Lipschitz transformations, we obtain 
\begin{align*}
d(gF^{+}, gx) \leq L_{g} \cdot \frac{\epsilon}{2L_{g}} = \frac{\epsilon}{2}.
\end{align*} Hence for $j$ sufficiently large, both $(\ref{omegafact1})$ and $(\ref{omegafact2})$ hold, and moreover
\begin{align*}
d(k_{g \gamma_{m_{j}}}P, gx) \leq d(k_{g \gamma_{m_{j}}} P, gF^{+}) + d(gF^{+}, gx) < \epsilon,
\end{align*} and
\begin{align*}
d_{\mathrm{Haus}}(\mathcal{Z}_{\ell_{g \gamma_{m_{j}}}^{-1}P^{-}}, \mathcal{Z}_{y}) &\leq d_{\mathrm{Haus}}(\mathcal{Z}_{\ell_{g \gamma_{m_{j}}}^{-1} P^{-}}, \mathcal{Z}_{F^{-}}) + d_{\mathrm{Haus}}(\mathcal{Z}_{F^{-}}, \mathcal{Z}_{y}) \\
&< \frac{\epsilon}{2} + \frac{\epsilon}{2L_{g}} < \epsilon.
\end{align*} This shows that $g \gamma_{m_{j}} \in \Gamma_{\mathcal{C}, gx, y, k_{m_{j}}, \epsilon}$ for all $j$ large enough, which is a contradiction. This completes the proof.  
\end{proof}
An entirely analogous argument also shows the following, so we omit its proof.
\begin{lem}
\label{inclusionofomegas2}
For every $g \in \Gamma$, $\epsilon > 0$, and $(x,y) \in \Lambda(\Gamma) \times \Lambda(\Gamma)^{-}$, there exist $N = N(x,y,g, \epsilon) \geq 1$ and $C = C(x,y,g, \epsilon) > 0$ so that
\begin{align*}
\Big( \Gamma_{\mathcal{C}', x, y, n + C, \frac{\epsilon}{2 L_{g^{-1}}}} \Big) \cdot g \subset \Gamma_{\mathcal{C}, x, g^{-1}y, n, \epsilon},
\end{align*} for all $n \geq N$.
\end{lem}
\begin{lem}
\label{limitsetcrosslimitset1}
There exists $(x_{0}, y_{0}) \in \Lambda(\Gamma) \times \Lambda(\Gamma)^{-}$ so that 
\begin{align*}
\sum_{\gamma \in \Gamma_{\mathcal{C}, x_{0},y_{0},n, \epsilon}} e^{-(\delta + r)||\kappa(\gamma)||} = \infty
\end{align*} for all $\epsilon > 0$ and $n \geq 1$.
\end{lem}
\begin{proof} 
For $x \in \mathcal{F}$, $\epsilon > 0$ and $n \geq 1$, define
\begin{align*}
\Gamma_{\mathcal{C}, x, n, \epsilon} := \{\gamma \in \Gamma_{\mathcal{C}} \ : \ ||\kappa(\gamma)|| \geq n, \ \ d(k_{\gamma}P, x) < \epsilon\}.
\end{align*}
We first show that there exists $x_{0} \in \Lambda(\Gamma)$ so that 
\begin{align}
\label{goal1}
\sum_{\gamma \in \Gamma_{\mathcal{C},x_{0},n, \epsilon}} e^{-(\delta + r) ||\kappa(\gamma)||} = \infty
\end{align} for all $\epsilon > 0$ and $n \geq 1$. Notice that since the limit set $\Lambda(\Gamma)$ is closed, it suffices to show that for all $m \geq 1$, there exists $x_{m} \in \Lambda (\Gamma)$ so that
\begin{align*}
\sum_{\gamma \in \Gamma_{\mathcal{C}, x_{m},n, \frac{1}{m}}} e^{-(\delta + r)||\kappa(\gamma)||} = \infty
\end{align*} for all $n \geq 1$. Indeed, if this holds, then the limit $x_{0} \in \mathcal{F}$ of any convergent subsequence $\{x_{m_{k}}\} \subset \{x_{m}\}$ is a point of $\Lambda(\Gamma)$ for which (\ref{goal1}) holds. 
\par 
Suppose for a contradiction that this does not hold. Then there exists an integer $m \geq 1$ so that for each $x \in \Lambda (\Gamma)$, we have
\begin{align}
\label{info1}
\sum_{\Gamma_{\mathcal{C}, x, \frac{1}{m}}} e^{-(\delta + r) ||\kappa(\gamma)||} < \infty, 
\end{align} where
\begin{align*}
\Gamma_{\mathcal{C}, x, \frac{1}{m}} := \{\gamma \in \Gamma_{\mathcal{C}} \ : \  d(k_{\gamma}P, x) < 1/m \}
\end{align*}
(note that we are using the discreteness of $\Gamma$ here). Let $E$ denote the closure of the set of accumulation points of $\{k_{\gamma}P : \gamma \in \Gamma_{\mathcal{C}}\}$. By definition it is a compact subset of $\Lambda(\Gamma)$. Thus there exist $x_{1}, \dots, x_{l} \in E$ so that $E \subset \bigcup_{i=1}^{l} B_{\frac{1}{m}}(x_{i})$. By construction, we have 
\begin{align}
\label{small}
\# \bigg( \Gamma_{\mathcal{C}} \setminus \bigg( \bigcup_{i=1}^{l} \Gamma_{\mathcal{C},x_{i}, \frac{1}{m}} \bigg) \bigg) < \infty.
\end{align} By (\ref{info1}), we know that 
\begin{align*}
\sum_{\Gamma_{\mathcal{C}, x_{i}, \frac{1}{m}}} e^{-(\delta + r) ||\kappa(\gamma)||} < \infty,
\end{align*} for all $1 \leq i \leq l$. Along with (\ref{small}), this implies that 
\begin{align*}
\sum_{\gamma \in \Gamma_{\mathcal{C}}} e^{-(\delta + r)||\kappa(\gamma)||} < \infty,
\end{align*} which contradicts (\ref{given1}). This shows that there exists some $x_{0} \in \Lambda(\Gamma)$ so that (\ref{goal1}) holds for all $\epsilon > 0$ and all $n \geq 1$. Now starting with this information, an entirely analogous argument shows that there exists $y_{0} \in \Lambda(\Gamma)^{-}$ so that 
\begin{align*}
\sum_{\gamma \in \Gamma_{\mathcal{C}, x_{0},y_{0},n,\epsilon}} e^{-(\delta + r)||\kappa(\gamma)||} = \infty
\end{align*} for all $\epsilon > 0$ and $n \geq 1$, as desired. 
\end{proof}
The following lemma shows that the conclusion of the previous lemma in fact holds for \emph{all} $(x,y) \in \Lambda(\Gamma) \times \Lambda(\Gamma)^{-}$.
\begin{lem}
\label{limitsetcrosslimitset2}
For all $(x,y) \in \Lambda(\Gamma) \times \Lambda(\Gamma)^{-}$, we have 
\begin{align*}
\sum_{\gamma \in \Gamma_{\mathcal{C}, x,y,n,\epsilon}} e^{-(\delta + r)||\kappa(\gamma)||} = \infty
\end{align*} for all $\epsilon > 0$ and $n \geq 1$.
\end{lem}
\begin{proof}
We first show that the conclusion of the lemma holds for all $(x,y_{0}) \in \Lambda(\Gamma) \times \{y_{0}\} \subset \Lambda(\Gamma) \times \Lambda(\Gamma)^{-}$. Then, for each fixed $x \in \Lambda(\Gamma)$, a very similar argument (using Lemma \ref{inclusionofomegas2} in place of Lemma \ref{inclusionofomegas}) shows that the lemma holds for all $(x,y) \in \{x\} \times \Lambda(\Gamma)^{-} \subset \Lambda(\Gamma) \times \Lambda(\Gamma)^{-}$, hence also for all pairs $(x,y) \in \Lambda(\Gamma) \times \Lambda(\Gamma)^{-}$. 
\par 
So now let $x \in \Lambda(\Gamma)$ and $\epsilon > 0$ be arbitrary. Since $\Gamma$ acts minimally on $\Lambda(\Gamma)$, there exists $g \in \Gamma$ so that $d(gx_{0}, x) < \epsilon/2$. By definition, we have
\begin{align*}
\Gamma_{\mathcal{C}, gx_{0}, y_{0}, n, \frac{\epsilon}{2}} \subset \Gamma_{\mathcal{C}, x, y_{0}, n, \epsilon}.
\end{align*} By Lemma \ref{inclusionofomegas}, there exist $N \geq 1$ and $C > 0$ so that 
\begin{align*}
g \cdot \Gamma_{\mathcal{C}', x_{0}, y_{0}, n+C, \frac{\epsilon}{4L_{g}}} \subset \Gamma_{\mathcal{C}, gx_{0}, y_{0}, n, \frac{\epsilon}{2}}
\end{align*} for all $n \geq N$. Fixing some $n_{0} \geq N$, we obtain
\begin{align*}
g \cdot \Gamma_{\mathcal{C}', x_{0}, y_{0}, n_{0}+C, \frac{\epsilon}{4L_{g}}} \subset \Gamma_{\mathcal{C}, x, y_{0}, n_{0}, \epsilon}. 
\end{align*} Along with Lemma \ref{limitsetcrosslimitset1} (recall that all the above lemmas hold for \emph{any} open cone in $\mathfrak{a}^{++}$ containing $u$, so in particular, for the cone $\mathcal{C}'$), this yields
\begin{align*}
\sum_{\gamma \in \Gamma_{\mathcal{C}, x, y_{0}, n_{0}, \epsilon}} e^{-(\delta + r) ||\kappa(\gamma)||} &\geq \sum_{\eta \in \Gamma_{\mathcal{C}', x_{0}, y_{0}, n_{0}+C, \frac{\epsilon}{4L_{g}}}} e^{-(\delta + r) ||\kappa(g \eta)||} \\
&\geq e^{-(\delta + r)||\kappa(g)||} \sum_{\eta \in \Gamma_{\mathcal{C}', x_{0}, y_{0}, n_{0}+C, \frac{\epsilon}{4L_{g}}}} e^{-(\delta + r) ||\kappa(\eta)||} = \infty. 
\end{align*} Lastly, for any $n \geq 1$, there are only finitely many elements of $\Gamma_{\mathcal{C}, x, y_{0}, n, \epsilon}$ that are not contained in $\Gamma_{\mathcal{C}, x, y_{0}, n_{0}, \epsilon}$, hence we also have 
\begin{align*}
\sum_{\gamma \in \Gamma_{\mathcal{C}, x, y_{0}, n, \epsilon}} e^{-(\delta + r) ||\kappa(\gamma)||} = \infty,
\end{align*} for all $n \geq 1$. By the discussion at the start of the proof, this concludes the argument.  
\end{proof}
The following result will be an essential ingredient in our construction of asymptotically large free subsemigroups of $\Gamma$. Indeed, it will enable us to use the notion of $\epsilon$-contracting elements introduced in section \ref{contractingelementssection}. It says that if $(x,y) \in \Lambda(\Gamma) \times \Lambda(\Gamma)^{-}$ is a transverse pair of limit points (that is, the flags $x \in \mathcal{F}$ and $y \in \mathcal{F}^{-}$ are transverse), then provided $\epsilon > 0$ is small enough and $n \geq 1$ is large enough, every element $\gamma \in \Gamma_{\mathcal{C},x,y,n,\epsilon}$ is $2 \epsilon$-contracting. Furthermore, it gives precise control on the location of $x_{\gamma}^{+} \in \mathcal{F}$ and $\mathcal{Z}_{x_{\gamma}^{-}} \subset \mathcal{F}$. 
\begin{prop}
\label{limitsetcrosslimitset3}
Let $(x,y) \in \Lambda(\Gamma) \times \Lambda(\Gamma)^{-}$ be a transverse pair of limit points. For all $0 < \epsilon < \frac{1}{8} d(x, \mathcal{Z}_{y})$, the following holds: there exists an integer $N= N(\epsilon) \geq 1$ such that if $n \geq N$, then $\gamma \in \Gamma_{\mathcal{C}, x,y,n,\epsilon}$ is $2 \epsilon$-contracting, 
\begin{align}
\label{attractingrepellingbounds}
d(x, x_{\gamma}^{+}) < 2 \epsilon \ \ \ \mathrm{and} \ \ \ d_{\mathrm{Haus}}(\mathcal{Z}_{y}, \mathcal{Z}_{x_{\gamma}^{-}}) < 2 \epsilon.
\end{align}
\end{prop}
\begin{proof}
We begin by showing that (\ref{attractingrepellingbounds}) holds. If not, then there exists a pair $(x,y) \in \Lambda(\Gamma) \times \Lambda(\Gamma)^{-}$ with $x \notin \mathcal{Z}_{y}$, a constant $0 < \epsilon < \frac{1}{8} d(x, \mathcal{Z}_{y})$, and a sequence $\{m_{n}\}_{n \geq 1} \subset \mathbb{N}$ with $m_{n} \rightarrow \infty$ and $\gamma_{n} \in \Gamma_{\mathcal{C}, x,y, m_{n}, \epsilon}$ so that either  
\begin{align*}
d(x,x_{\gamma_{n}}^{+}) \geq 2 \epsilon \ \ \mathrm{or} \ \ d_{\mathrm{Haus}}(\mathcal{Z}_{y}, \mathcal{Z}_{x_{\gamma_{n}}^{-}}) \geq 2 \epsilon,
\end{align*}
for all $n \geq 1$. After passing to a subsequence if necessary, we can assume without loss of generality that 
\begin{align}
\label{convergencetotransverseflags}
k_{\gamma_{n}}P \rightarrow F^{+} \in \mathcal{F} \ \ \mathrm{and} \ \ \ell_{\gamma_{n}}^{-1}P^{-} \rightarrow F^{-} \in \mathcal{F}^{-},
\end{align} hence also $\mathcal{Z}_{\ell_{\gamma_{n}}^{-1}P^{-}} \rightarrow \mathcal{Z}_{F^{-}}$ in the Hausdorff topology. By definition, we have 
\begin{align*}
d(x, k_{\gamma_{n}}P) < \epsilon \ \ \mathrm{and} \ \  d_{\mathrm{Haus}}(\mathcal{Z}_{y}, \mathcal{Z}_{\ell_{\gamma_{n}}^{-1}P^{-}}) < \epsilon,
\end{align*} for all $n \geq 1$, so 
\begin{align}
\label{close1}
d(x, F^{+}) \leq \epsilon \ \ \mathrm{and} \ \ d_{\mathrm{Haus}}(\mathcal{Z}_{y}, \mathcal{Z}_{F^{-}}) \leq \epsilon.
\end{align} Since
\begin{align*}
d(F^{+}, \mathcal{Z}_{F^{-}}) \geq d(x, \mathcal{Z}_{y}) - d(x, F^{+}) - d_{\mathrm{Haus}}(\mathcal{Z}_{F^{-}}, \mathcal{Z}_{y}) > 8 \epsilon - \epsilon - \epsilon = 6 \epsilon > 0,
\end{align*} we have that $F^{+}$ and $F^{-}$ are transverse. Notice that, by construction, we have $\min_{\alpha \in \Delta} \alpha(\kappa(\gamma_{n})) \rightarrow \infty$ as $n \rightarrow \infty$. Recalling (\ref{convergencetotransverseflags}), we use Proposition \ref{attractingrepellinglocation} to conclude that
\begin{align}
\label{close2}
d(F^{+}, x_{\gamma_{n}}^{+}) < \epsilon \ \  \mathrm{and} \ \ d_{\mathrm{Haus}}(\mathcal{Z}_{F^{-}}, \mathcal{Z}_{x_{\gamma_{n}}^{-}}) < \epsilon, 
\end{align} for all $n$ sufficiently large. Combining (\ref{close1}) and (\ref{close2}), it follows that, 
\begin{align*}
d(x, x_{\gamma_{n}}^{+}) &\leq d(x, F^{+}) + d(F^{+}, x_{\gamma_{n}}^{+}) < 2 \epsilon, 
\end{align*} 
and
\begin{align*}
d_{\mathrm{Haus}}(\mathcal{Z}_{y}, \mathcal{Z}_{x_{\gamma_{n}}^{-}}) &\leq d_{\mathrm{Haus}}(\mathcal{Z}_{y}, \mathcal{Z}_{F^{-}}) + d_{\mathrm{Haus}}(\mathcal{Z}_{F^{-}}, \mathcal{Z}_{x_{\gamma_{n}}^{-}}) < 2 \epsilon,
\end{align*} for all $n$ sufficiently large, which is a contradiction. 
\par 
To summarize, given a transverse pair $(x,y) \in \Lambda(\Gamma) \times \Lambda(\Gamma)^{-}$, and given $0 < \epsilon < \frac{1}{8} d(x, \mathcal{Z}_{y})$, there exists $N = N(\epsilon) \geq 1$ so that for all $n \geq N$, every $\gamma \in \Gamma_{\mathcal{C}, x, y,n, \epsilon}$ satisfies (\ref{attractingrepellingbounds}). We will use this to show that, for all (possibly bigger) $n$ large enough, each element of $\Gamma_{\mathcal{C}, x, y,n, \epsilon}$ is $2 \epsilon$-contracting. 
\par 
Let $n \geq N$ and $\gamma \in \Gamma_{\mathcal{C}, x, y,n, \epsilon}$ be arbitrary. By (\ref{attractingrepellingbounds}) and the definition of $\epsilon_{0}$, we have
\begin{align}
\label{attrepdist}
d(x_{\gamma}^{+}, \mathcal{Z}_{x_{\gamma}^{-}}) &\geq d(x, \mathcal{Z}_{y}) - d_{\mathrm{Haus}}(\mathcal{Z}_{y}, \mathcal{Z}_{x_{\gamma}^{-}}) - d(x, x_{\gamma}^{+}) > 8 \epsilon - 2 \epsilon - 2 \epsilon = 4 \epsilon.
\end{align} Suppose for a contradiction that there exists a sequence $\{m_{n}\}_{n \geq N} \subset \mathbb{N}$ with $m_{n} \to \infty$ and $\gamma_{n} \in \Gamma_{\mathcal{C}, x ,y, m_{n}, \epsilon}$ which are not $2 \epsilon$-contracting. As before, after passing to a subsequence if necessary, we may assume without loss of generality that 
\begin{align*}
\min_{\alpha \in \Delta} \alpha(\kappa(\gamma_{n})) \rightarrow \infty, 
 \ \ k_{\gamma_{n}}P \rightarrow F^{+} \in \mathcal{F}, \ \ \mathrm{and} \ \ \ell_{\gamma_{n}}^{-1}P^{-} \rightarrow F^{-} \in \mathcal{F}^{-}. 
\end{align*}
By construction, we have 
\begin{align*}
d(F^{+}, x) \leq \epsilon \ \ \mathrm{and} \ \ d_{\mathrm{Haus}}(\mathcal{Z}_{y}, \mathcal{Z}_{F^{-}}) \leq \epsilon.
\end{align*}
Arguing as above, we see that $F^{+}$ and $F^{-}$ are transverse, whence Proposition \ref{attractingrepellinglocation} implies that $x_{\gamma_{n}}^{+} \rightarrow F^{+}$ and $\mathcal{Z}_{x_{\gamma_{n}}^{-}} \rightarrow \mathcal{Z}_{F^{-}}$ in the Hausdorff topology. But then Proposition \ref{northsouthflagvars} yields
\begin{align*}
\gamma_{n} \big(\mathcal{F} \smallsetminus N_{2 \epsilon} \big(\mathcal{Z}_{x_{\gamma_{n}}^{-}} \big) \big) \subset \gamma_{n} \big(\mathcal{F} \smallsetminus N_{\epsilon} \big(\mathcal{Z}_{F^{-}} \big) \big) \subset \overline{B_{\epsilon}(F^{+})} \subset \overline{B_{2 \epsilon}(x_{\gamma_{n}}^{+})}, 
\end{align*} for all $n$ sufficiently large. Moreover, the restriction $\gamma_{n}|_{\mathcal{F} \smallsetminus N_{2 \epsilon} (\mathcal{Z}_{x_{\gamma_{n}}^{-}})}$ is $2 \epsilon$-Lipschitz. Along with (\ref{attrepdist}), this shows that $\gamma_{n}$ is $2 \epsilon$-contracting for all $n$ large enough, which is a contradiction. This concludes the proof of the proposition.  
\end{proof}
As in the above proposition, let $(x,y) \in \Lambda(\Gamma) \times \Lambda(\Gamma)^{-}$ be a transverse pair of limit points. In order to establish that the free semigroups we will construct are also Zariski dense, we will first need to show that the sets $\Gamma_{\mathcal{C},x,y,n,\epsilon}$ are Zariski dense for all $n \geq 1$ and $\epsilon > 0$ sufficiently small. We first need some more notation. For an open cone $\mathcal{V} \subset \mathfrak{a}^{++}$, $n \geq 1$, and $\epsilon > 0$ sufficiently small, define
\begin{align*}
&\Xi_{\mathcal{V},x,y,\epsilon} := \{g \in \Gamma : \lambda(g) \in \mathcal{V}, \ d(x_{g}^{+}, x) < \epsilon, \ \mathrm{and} \ d_{\mathrm{Haus}}(\mathcal{Z}_{x_{g}^{-}}, \mathcal{Z}_{y}) < \epsilon \}, \\
&\Xi_{\mathcal{V},x,y,n,\epsilon} := \{g \in \Gamma : \lambda(g) \in \mathcal{V}, \ ||\lambda(g)|| \geq n, \ d(x_{g}^{+}, x) < \epsilon, \ \mathrm{and} \ d_{\mathrm{Haus}}(\mathcal{Z}_{x_{g}^{-}}, \mathcal{Z}_{y}) < \epsilon \},
\end{align*}
and
\begin{align*}
\Upsilon_{\mathcal{V},x,y,n,\epsilon} := \{g \in \Gamma_{\mathrm{lox}} : \kappa(g) \in \mathcal{V}, \ ||\kappa(g)|| \geq n, \ d(x_{g}^{+}, x) < \epsilon, \ \mathrm{and} \ d_{\mathrm{Haus}}(\mathcal{Z}_{x_{g}^{-}}, \mathcal{Z}_{y}) < \epsilon \}.
\end{align*}
The following result, while not formulated in exactly this way in Benoist's work \cite{B2}, follows immediately by unraveling the definitions in his result (in fact, Benoist proves an even stronger result, but the following less general version is sufficient for our purposes).
\begin{lem} [Lemma 4.2 of \cite{B2}]
\label{ZariskidensityJordan}
Let $\mathcal{V} \subset \mathfrak{a}^{++}$ be an open cone intersecting the limit cone $\mathcal{L}_{\Gamma}$ and let $(x,y) \in \Lambda(\Gamma) \times \Lambda(\Gamma)^{-}$ be a transverse pair of limit points. Then for all $n \geq 1$ and $\epsilon > 0$ sufficiently small, the subset $\Xi_{\mathcal{V},x,y,n,\epsilon}$ of $\Gamma$ is still Zariski dense in $G$.
\end{lem}
Before we can apply this result, we need a few more observations.
\begin{lem}
\label{JordanCartanforBenoistsset}
For every $\epsilon > 0$, there exists a symmetric compact set $V = V(\epsilon) \subset \mathfrak{a}$ (i.e., $V = -V$) so that, for any loxodromic element $g \in G$ such that $d(x_{g}^{+},x) \leq \epsilon$ and $d_{\mathrm{Haus}}(\mathcal{Z}_{x_{g}^{-}}, \mathcal{Z}_{y}) \leq \epsilon$, we have
\begin{align*}
\kappa(g) - \lambda(g) \in V.
\end{align*}
\end{lem}
\begin{proof}
First notice that there exists a compact set $Q' \subset G \cdot (P,P^{-})$, depending only on $\epsilon > 0$, so that, for any $g \in G$ satisfying the hypotheses of the lemma, we have $(x_{g}^{+}, x_{g}^{-}) \in Q'$. Thus, there exists a compact subset $Q \subset G$ so that for all $(a,b) \in Q'$, there exists $h \in Q$ such that $ha = P$ and $hb = P^{-}$; moreover, enlarging $Q$ if necessary, we may assume that $Q = Q^{-1} := \{q^{-1} : q \in Q\}$. In particular, we see that for each $g \in G$, there exists $h_{g} \in Q$ so that $h_{g} x_{g}^{+} = P$ and $h_{g} x_{g}^{-} = P^{-}$. Therefore $h_{g}gh_{g}^{-1}$ fixes $(P,P^{-}) \in \mathcal{F} \times \mathcal{F}^{-}$, hence 
\begin{align*}
h_{g}gh_{g}^{-1} = m_{g} a_{g}
\end{align*} for some $m_{g} \in M$ and $a_{g} \in A$. Writing $g = h_{g}^{-1} m_{g} a_{g} h_{g}$, Lemma \ref{Yves'Estimate} implies that there is a compact subset $V \subset \mathfrak{a}$ (depending only on $QM$, hence only on $\epsilon > 0$) so that
\begin{align*}
\kappa (g) - \kappa (a_{g}) \in V.
\end{align*} Enlarging $V$ if necessary, we can assume that $V = -V$. Since $M$ is the centralizer of $A$ in $K$, we also have $\lambda(g) = \kappa(a_{g})$. Hence,
\begin{align*}
\kappa(g) - \lambda(g) \in V,
\end{align*} as desired.
\end{proof}
\begin{lem}
\label{InclusionCartanSets}
For all $n \geq 1$ sufficiently large and $\epsilon > 0$ small enough, we have
\begin{align*}
\Upsilon_{\mathcal{V},x,y,n,\epsilon} \subset \Gamma_{\mathcal{V},x,y,n, 2 \epsilon}.
\end{align*}
\end{lem}
\begin{proof}
Recall first that, by Proposition \ref{limitsetcrosslimitset3}, the elements of $\Gamma_{\mathcal{V},x,y,n, 2 \epsilon}$ are in particular loxodromic provided $n \geq 1$ is large enough and $\epsilon > 0$ is small enough. If the lemma does not hold, then there exists a sequence $\{m_{n}\}$ of integers with $m_{n} \rightarrow \infty$ and a sequence of elements $\{g_{n}\}$ with $g_{n} \in \Upsilon_{\mathcal{V},x,y,n, \epsilon}$ so that either $d(k_{g_{n}}P, x) \geq 2 \epsilon$ or $d_{\mathrm{Haus}}(\mathcal{Z}_{\ell_{g_{n}}^{-1}P^{-}}, \mathcal{Z}_{y}) \geq 2 \epsilon$ for all $n$. Without loss of generality, after passing to a subsequence $\{g_{n_{j}}\}$, we may assume that  
\begin{align*}
d(k_{g_{n_{j}}}P, x) \geq 2 \epsilon \ \ \mathrm{for \ all} \ \ j \geq 1, \ \ k_{g_{n_{j}}}P \rightarrow F^{+} \in \mathcal{F}, \ \ \ell_{g_{n_{j}}}^{-1}P^{-} \rightarrow F^{-},
\end{align*} and also $\min_{\alpha \in \Delta} \alpha (\kappa(g_{n_{j}})) \rightarrow \infty$ (a similar argument applies if instead we assume $d_{\mathrm{Haus}}(\mathcal{Z}_{\ell_{g_{n_{j}}}^{-1}P^{-}}, \mathcal{Z}_{y}) \geq 2 \epsilon$ for all $j \geq 1$). But then $d(F^{+},x) \geq 2 \epsilon$, hence by Proposition \ref{northsouthflagvars} we have $d(x_{g_{n_{j}}}^{+},x) > \frac{3 \epsilon}{2}$ for all $j$ sufficiently large. This contradicts the fact that $g_{n_{j}} \in \Upsilon_{\mathcal{V},x,y,n_{j}, \epsilon}$, which concludes the proof. 
\end{proof}
We now prove that the sets $\Gamma_{\mathcal{C},x,y,n,\epsilon}$ are Zariski dense for all $n \geq 1$ sufficiently large and $\epsilon > 0$ small enough.
\begin{prop}
\label{ZariskidensityCartanset}
For all $n \geq 1$ sufficiently large and $\epsilon > 0$ sufficiently small, the subset $\Gamma_{\mathcal{C},x,y,n,\epsilon}$ of $\Gamma$ is still Zariski dense in $G$. 
\end{prop}
\begin{proof}
Fix an open cone $\mathcal{V} \subset \mathfrak{a}^{++}$ containing $u \in \mathfrak{a}^{++}$ and such that $\overline{\mathcal{V}} \subsetneq \mathcal{C}$. By Lemma \ref{ZariskidensityJordan} and Lemma \ref{InclusionCartanSets}, for $n \geq 1$ large enough and $\epsilon > 0$ small enough, we know that:
\begin{itemize}
    \item[(1)] The set $\Xi_{\mathcal{V},x,y,n, \frac{\epsilon}{2}}$ is Zariski dense in $G$.
    \item[(2)] We have the inclusion of sets $\Upsilon_{\mathcal{C},x,y,n, \frac{\epsilon}{2}} \subset \Gamma_{\mathcal{C},x,y,n, \epsilon}$.
\end{itemize} By Lemma \ref{JordanCartanforBenoistsset}, there exists $C = C(\epsilon/2) > 0$ so that $||\lambda(g) - \kappa (g)|| \leq C$ for all $g \in \Xi_{\mathcal{V},x,y, \frac{\epsilon}{2}}$. Since $\overline{\mathcal{V}}$ is a proper subset of $\mathcal{C}$, we obtain
\begin{align*}
\Xi_{\mathcal{V},x,y,n, \frac{\epsilon}{2}} \subset \Upsilon_{\mathcal{C},x,y,n, \frac{\epsilon}{2}} \subset \Gamma_{\mathcal{C},x,y,n,\epsilon}, 
\end{align*} provided $n \geq 1$ is large enough. By item (1) above, it follows that $\Gamma_{\mathcal{C},x,y,n,\epsilon}$ is Zariski dense in $G$, as desired.
\end{proof}
\begin{rem}
\label{afortiori}
Since $\Gamma_{\mathcal{C},x,y,n_{1}, \epsilon_{1}} \subset \Gamma_{\mathcal{C},x,y,n_{2},\epsilon_{2}}$ when $n_{1} \geq n_{2} \geq 1$ and $0 < \epsilon_{1} \leq \epsilon_{2}$, Proposition \ref{ZariskidensityCartanset} in fact holds for all $n \geq 1$ and all $\epsilon > 0$.
\end{rem}

The elements in the generating sets of our free semigroups will come from examining certain annular regions of the sets $\Gamma_{\mathcal{C},x,y, n, \epsilon}$. Namely, we will study sets of the form
\begin{align*}
\mathcal{A}_{\mathcal{C},x,y,n,w,\epsilon} := \{\gamma \in \Gamma_{\mathcal{C},x,y,n,\epsilon} : n \leq ||\kappa(\gamma)|| < n+w \},
\end{align*} where $n,w \geq 1$ are integers.
\begin{lem}
\label{AnnularPoincDiv}
For any $0 < \delta_{0} < \delta + r$ and any integer $w \geq 1$, we have
\begin{align*}
\limsup_{n \to \infty} \sum_{\gamma \in \mathcal{A}_{\mathcal{C},x,y,n,w,\epsilon}} e^{-\delta_{0} ||\kappa(\gamma)||} = \infty.
\end{align*}
\end{lem}
\begin{proof}
Suppose to the contrary that there exists some $0 < \delta_{0} < \delta + r$ and $w \geq 1$ so that
\begin{align*}
\limsup_{n \to \infty} \sum_{\gamma \in \mathcal{A}_{\mathcal{C},x,y,n,w,\epsilon}} e^{-\delta_{0} ||\kappa(\gamma)||} < \infty.
\end{align*} This implies that there exists a constant $M > 0$ so that 
\begin{align*}
\sum_{\gamma \in \mathcal{A}_{\mathcal{C},x,y,n,w,\epsilon}} e^{-\delta_{0} ||\kappa(\gamma)||} \leq M
\end{align*} for all $n \geq 1$. Hence,
\begin{align*}
\sum_{\gamma \in \Gamma_{\mathcal{C}, x, y, w, \epsilon}} e^{-(\delta + r) ||\kappa(\gamma)||} &= \sum_{n=1}^{\infty} \sum_{\gamma \in \mathcal{A}_{\mathcal{C},x,y,nw,w,\epsilon}} e^{-(\delta + r)||\kappa(\gamma)||} \\
&\leq \sum_{n=1}^{\infty} \bigg( e^{-(\delta + r - \delta_{0})nw} \sum_{\gamma \in \mathcal{A}_{\mathcal{C},x,y,nw,w,\epsilon}} e^{-\delta_{0} ||\kappa(\gamma)||} \bigg) \\
&\leq M \sum_{n=1}^{\infty} e^{-w(\delta + r - \delta_{0}) n} < \infty,
\end{align*} which contradicts Lemma \ref{limitsetcrosslimitset2}. This concludes the proof.
\end{proof}
\section{Relations Between Two Different Notions of Shadows}
\label{symspaceshadowssection}
It will be important in our main construction to be able to relate the shadows we have defined for $\epsilon$-contracting elements to another notion of higher-rank shadows which has already been used extensively in the literature (see for instance \cite{KOW1}, \cite{KOW2}, and \cite{LO}). Recall that $X = G/K$ denotes the symmetric space of $G$. In what follows, for $q \in X$ and $R > 0$, we let $B_{R}(q) := \{x \in X : d_{X}(x,q) < R\}$.
\begin{defn} [symmetric space shadows]
\label{symspshadows}
For $p \in X$, the \emph{shadow} $\mathcal{O}_{R}(p,q) \subset \mathcal{F}$ \emph{of the ball} $B_{R}(q)$ \emph{viewed from} $p$ is defined to be
\begin{align*}
\mathcal{O}_{R}(p,q) := \{gP \in \mathcal{F} : g \in G, \ go = p, \ \mathrm{and} \ gA^{+}o \cap B_{R}(q) \neq \emptyset \}.
\end{align*} We also define the shadow $\mathcal{O}_{R}(\eta, p) \subset \mathcal{F}$, viewed from $\eta \in \mathcal{F}^{-}$, by
\begin{align*}
\mathcal{O}_{R}(\eta,p) := \{gP \in \mathcal{F} : g \in G, \ gk_{0}P^{-} = \eta, \ \mathrm{and} \ go \in B(p,R)\}.
\end{align*}
\end{defn}
Notice that, by definition, every point of $\mathcal{O}_{R}(\eta, p)$ is transverse to $\eta$, that is, $\mathcal{O}_{R}(\eta, p) \subset \mathcal{F} \smallsetminus \mathcal{Z}_{\eta}$. We will use the following result of Lee--Oh.
\begin{prop} [Lemma 5.6 of \cite{LO}]
\label{continuityonviewpoints}
Let $p \in X$, $\eta \in \mathcal{F}^{-}$, and $R > 0$ be arbitrary. If $\{g_{n}\} \subset G$ is such that $\alpha(\kappa(g_{n})) \rightarrow \infty$ for all $\alpha \in \Delta$ and $k_{g_{n}}P^{-} \rightarrow \eta$, then for all $0 < \epsilon < R$, we have 
\begin{align*}
\mathcal{O}_{R - \epsilon}(g_{n}o,p) \subset \mathcal{O}_{R}(\eta, p) \subset \mathcal{O}_{R + \epsilon}(g_{n}o, p),
\end{align*} for all $n$ sufficiently large.
\end{prop}
Using this result, we are able to deduce that for well-behaved sequences of elements $\{g_{n}\}_{n \geq 1}$ in $G$, the diameters of shadows of balls centered at $g_{n}o \in X$ and of uniform radii tend to zero as $n \rightarrow \infty$. The precise statement is the following:
\begin{lem}
\label{diametertozero} 
Let $\{g_{n}\}$ be a sequence in $G$ such that $\alpha(\kappa(g_{n})) \rightarrow \infty$ for all $\alpha \in \Delta$, $k_{g_{n}}P \rightarrow F^{+} \in \mathcal{F}$, and $\ell_{g_{n}}^{-1}P^{-} \rightarrow F^{-} \in \mathcal{F}^{-}$. Then, for any $R > 0$, as $n \rightarrow \infty$ we have
\begin{align*}
\mathcal{O}_{R}(o,g_{n}o) \rightarrow F^{+}
\end{align*} in the Hausdorff topology. In particular,
\begin{align*}
\lim_{n \to \infty} \mathrm{diam} \ \mathcal{O}_{R}(o, g_{n} o) = 0.
\end{align*}
\end{lem}
\begin{proof}
Fix $\epsilon > 0$ and notice first that $\ell_{g_{n}}^{-1} P^{-} \rightarrow F^{-}$ is equivalent to the statement that $k_{g_{n}^{-1}}P^{-} \rightarrow F^{-}$. Since the opposition involution preserves the set $\Delta$ of simple restricted roots, we have 
\begin{align*}
\min_{\alpha \in \Delta} \alpha(\kappa(g_{n}^{-1})) = \min_{\alpha \in \Delta}(\alpha(\kappa(g_{n})) \rightarrow \infty,
\end{align*} as $n \to \infty$. By Proposition \ref{continuityonviewpoints}, we have
\begin{align*}
\mathcal{O}_{R}(g_{n}^{-1}o,o) \subset \mathcal{O}_{R+\epsilon}(F^{-}, o),
\end{align*} for all $n$ large enough. Recall (see the comment following Definition \ref{symspshadows}) that 
\begin{align*}
\mathcal{O}_{R+\epsilon}(F^{-},o) \subset \mathcal{F} \smallsetminus \mathcal{Z}_{F^{-}}.
\end{align*} By Proposition \ref{northsouthflagvars}, we obtain
\begin{align*}
\mathcal{O}_{R}(o, g_{n}o) = g_{n} \cdot \mathcal{O}_{R}(g_{n}^{-1}o,o) \subset g_{n} \cdot \mathcal{O}_{R+\epsilon}(F^{-},o) \rightarrow F^{+},
\end{align*} where the convergence is in the Hausdorff topology (the above inclusion only holds for $n$ sufficiently large, but that suffices in order to establish the convergence to $F^{+}$). In particular,
\begin{align*}
\lim_{n \to \infty} \mathrm{diam} \ \mathcal{O}_{R}(o, g_{n} o) = 0,
\end{align*} which concludes the proof.
\end{proof}

The following is a special case of a result of Kim--Zimmer.
\begin{lem} [Lemma 9.10 in \cite{KZ1}]
\label{relcptshadowinclusion}
For any relatively compact subset $V \subset N^{-}$, there exists $R > 0$ so that, if $g \in G$ has a Cartan decomposition $g = ka \ell \in KA^{+}K$, then 
\begin{align*}
\ell^{-1}VP \subset \mathcal{O}_{R}(g^{-1}o, o).
\end{align*}
\end{lem}
Using this lemma, we can show that the shadows of $\epsilon$-contracting elements (in the sense of Definition \ref{contractingshadow}) whose Cartan projections are sufficiently deep inside of the positive Weyl chamber are included inside of symmetric space shadows of balls of uniform radii. 
\begin{lem}
\label{contractsymspaceinclusion}
Let $(x,y) \in \Lambda(\Gamma) \times \Lambda(\Gamma)^{-}$ be such that $x \notin \mathcal{Z}_{y}$. For all $\epsilon > 0$ sufficiently small, there exists $M = M(\epsilon) \geq 1$ and $R = R(\epsilon) > 0$ so that, if $n \geq M$ and $\gamma \in \Gamma_{\mathcal{C}, x, y,n, \epsilon}$, then
\begin{align*}
\mathcal{S}_{2 \epsilon}(\gamma) \subset \mathcal{O}_{R}(o, \gamma o). 
\end{align*}
\end{lem}
\begin{proof}
Note that the desired inclusion is equivalent to $\gamma^{-1}\mathcal{S}_{2 \epsilon}(\gamma) \subset \mathcal{O}_{R}(\gamma^{-1}o, o)$. For $\gamma \in \Gamma$, let $\gamma = k_{\gamma}a_{\gamma}\ell_{\gamma} \in KA^{+}K$ be a Cartan decomposition. Let $0 < \epsilon < \frac{1}{8} d(x, \mathcal{Z}_{y})$. By Proposition \ref{limitsetcrosslimitset3} and its proof, there exists $M = M(\epsilon)$ so that if $n \geq M$, then any $\gamma \in \Gamma_{\mathcal{C},x,y,n,\epsilon}$, then $\gamma$ is $2 \epsilon$-contracting and 
\begin{align*}
d_{\mathrm{Haus}}(\mathcal{Z}_{x_{\gamma}^{-}}, \mathcal{Z}_{\ell_{\gamma}^{-}P^{-}}) < \epsilon.
\end{align*} Hence 
\begin{align*}
\gamma^{-1}S_{2\epsilon}(\gamma) = \mathcal{F} \smallsetminus N_{2\epsilon} (\mathcal{Z}_{x_{\gamma}^{-}}) \subset \mathcal{F} \smallsetminus N_{\epsilon}(\mathcal{Z}_{\ell_{\gamma}^{-1}P^{-}}).
\end{align*} Since $K$ acts by isometries on $\mathcal{F}$, it therefore suffices to show that there exists $R > 0$, depending only on $\epsilon$, so that
\begin{align*}
\ell_{\gamma}^{-1} \big( \mathcal{F} \smallsetminus N_{\epsilon}(\mathcal{Z}_{P^{-}}) \big) \subset \mathcal{O}_{R}(\gamma^{-1}o, o).
\end{align*} By Lemma \ref{relcptshadowinclusion}, we just need to show that there is a relatively compact subset $V \subset N^{-}$, depending only on $\epsilon$, so that $\mathcal{F} \smallsetminus N_{\epsilon}(\mathcal{Z}_{P^{-}}) \subset VP$. To see why this holds, recall that the exponential map gives a diffeomorphism 
\begin{align*}
\exp : \mathfrak{n}^{-} \rightarrow N^{-} = \exp(\mathfrak{n}^{-}).
\end{align*}
Moreover, the \emph{Langlands decomposition} states that the map 
\begin{align*}
N^{-} \times L &\rightarrow P^{-}, \\
(n,\ell) &\mapsto n \ell
\end{align*} is a diffeomorphism, where we recall that $L := P \cap P^{-}$ is the Levi subgroup (see for instance Theorem $1.2.4.8$ of \cite{Wa}). It follows that $N^{-}$ acts simply transitively on $\mathcal{F} \smallsetminus \mathcal{Z}_{P^{-}}$, hence the map 
\begin{align*}
T : \mathfrak{n}^{-} &\rightarrow \mathcal{F} \smallsetminus \mathcal{Z}_{P^{-}}, \\
X &\mapsto e^{X}P
\end{align*} is a diffeomorphism. Thus
\begin{align*}
V := \exp \big(T^{-1} \big(\mathcal{F} \smallsetminus N_{\epsilon}(\mathcal{Z}_{P^{-}}) \big) \big)
\end{align*}
is a compact subset of $N^{-}$, depending only on $\epsilon$, satisfying
\begin{align*}
VP = \mathcal{F} \smallsetminus N_{\epsilon}(\mathcal{Z}_{P^{-}}).
\end{align*}
This concludes the proof.
\end{proof}
Before proceeding further, we record some elementary observations about symmetric space shadows.
\begin{lem}
\label{shadowobservations} 
For all $\gamma, \gamma_{1}, \gamma_{2} \in G$ and $R > 0$, the following hold:
\begin{itemize}
    \item[(1)] If we have $\ell P \in \mathcal{O}_{R}(o, \gamma o)$ for some $\ell \in K$, then 
    \begin{align*}
        d_{X} \big(\ell e^{\kappa(\gamma)} o, \gamma o \big) \leq 2 R.
    \end{align*}
    \item[(2)] If 
    $\mathcal{O}_{R}(o, \gamma_{1} o) \cap \mathcal{O}_{R}(o, \gamma_{2} o) \neq \emptyset$, then 
    \begin{align*}
    d_{X}(\gamma_{1}o, \gamma_{2}o) \leq 4R + ||\kappa(\gamma_{1}) - \kappa(\gamma_{2})||.
    \end{align*}
\end{itemize}
\end{lem}
\begin{proof}[Proof of (1)]
If $\ell P \in \mathcal{O}_{R}(o, \gamma o)$, then by definition there exists $H \in \mathfrak{a}^{+}$ so that
\begin{align*}
\big| \big|\kappa(\gamma^{-1} \ell e^{H}) \big| \big| = d_{X} \big(\ell e^{H}o, \gamma o \big) \leq R.
\end{align*} Then by Lemma \ref{differencecartanprojections}, we obtain
\begin{align*}
||H - \kappa(\gamma)|| = \big| \big|\kappa(\ell e^{H}) - \kappa(\gamma) \big| \big| = \big| \big|\kappa \big(\gamma \cdot \big(\gamma^{-1} \ell e^{H} \big) \big) - \kappa(\gamma) \big| \big| \leq \big| \big|\kappa(\gamma^{-1} \ell e^{H}) \big| \big| \leq R.
\end{align*} Hence,
\begin{align*}
d_{X} \big(\ell e^{\kappa(\gamma)}o, \gamma o \big) \leq d_{X} \big(\ell e^{\kappa(\gamma)}o, \ell e^{H} o \big) + d_{X} \big(\ell e^{H}o, \gamma o \big) \leq ||\kappa(\gamma) - H|| + R \leq 2R,
\end{align*} which is what we wanted to show.
\end{proof}
\begin{proof} [Proof of (2).]
By assumption, we have 
\begin{align*}
\ell P \in \mathcal{O}_{R}(o, \gamma_{1} o) \cap \mathcal{O}_{R}(o, \gamma_{2} o),
\end{align*}
for some $\ell \in K$. Then part (1) implies that 
\begin{align*}
d_{X} \big( \ell e^{\kappa(\gamma_{1})} o, \gamma_{1} o \big) \leq 2 R \ \ \mathrm{and} \ \ d_{X} \big( \ell e^{\kappa(\gamma_{2})} o, \gamma_{2} o \big) \leq 2 R.
\end{align*}
Hence,
\begin{align*}
d_{X}(\gamma_{1} o, \gamma_{2} o) &\leq d_{X} \big( \gamma_{1}o, \ell e^{\kappa(\gamma_{1})} o \big) + d_{X} \big( \ell e^{\kappa(\gamma_{1})} o, \ell e^{\kappa(\gamma_{2})} o \big) + d_{X} \big( \ell e^{\kappa(\gamma_{2})} o, \gamma_{2} o \big) \\
&\leq 4R + ||\kappa(\gamma_{1}) - \kappa(\gamma_{2})||,
\end{align*} as desired. 
\end{proof}
As before, let $0 < \delta < \delta + r < \delta (\Gamma)$ and let $u \in \mathfrak{a}^{++}$ be a unit vector so that $\psi_{\Gamma}(u) > \delta + r$. Notice that, for any open cone $\mathcal{C} \subset \mathfrak{a}^{++}$ containing $u$, there exists a constant $\lambda = \lambda(\mathcal{C}) > 0$ such that: if $H_{1}, H_{2} \in \mathcal{C}$, $n,w \geq 1$ are any integers, and $||H_{1}||, ||H_{2}|| \in [n,n+w)$, then 
\begin{align}
\label{annularbound}
||H_{1} - H_{2}|| \leq \lambda (n+w).
\end{align} 
Moreover,
\begin{align*}
\inf_{\mathcal{C} \ni u} \lambda(\mathcal{C}) = 0,
\end{align*}
where the infimum is taken over all open cones $\mathcal{C} \subset \mathfrak{a}^{++}$ containing $u$. Hence, we may fix an open cone $\mathcal{C} \subset \mathfrak{a}^{++}$ containing $u$ for which $\lambda = \lambda(\mathcal{C}) > 0$ is sufficiently small so that
\begin{align}
\label{needed1}
0 < \delta < \delta (1 + \lambda) < \delta_{0} < \delta + r < \delta (\Gamma),
\end{align} and also
\begin{align}
\label{needed2}
(\delta(\Gamma) + 1)\lambda - \bigg( \frac{\delta_{0}}{1+\lambda} - \delta \bigg) < 0.
\end{align} These ad hoc inequalities come up naturally during the proof of Proposition \ref{annupoincdivwithdisjtshadows}, which is essentially a refinement of Lemma \ref{AnnularPoincDiv} in that we obtain the same conclusion while only summing over elements of the annular regions which have pairwise disjoint symmetric space shadows. To make this precise, first fix a transverse pair of limit points $(x,y) \in \Lambda(\Gamma) \times \Lambda(\Gamma)^{-}$ let $0 < \epsilon < \frac{1}{8} d(x, \mathcal{Z}_{y})$ and $R > 0$ be arbitrary, and 
let $w \geq 1$ be an arbitrary integer. By Lemma \ref{AnnularPoincDiv} and (\ref{needed1}), we have
\begin{align}
\label{annulardivergence}
\limsup_{n \to \infty} \sum_{\sigma \in \mathcal{A}_{\mathcal{C},x,y,n,w,\epsilon}} e^{-\delta_{0}||\kappa(\sigma)||} = \infty.
\end{align}
By part (2) of Lemma \ref{shadowobservations} and (\ref{annularbound}), we obtain
\begin{align*}
&\# \{ \gamma \in \mathcal{A}_{\mathcal{C},x,y,n,w,\epsilon} : \mathcal{O}_{R}(o, \gamma o) \cap \mathcal{O}_{R}(o, \eta o) \neq \emptyset \} \\
&\leq \# \{\gamma \in \Gamma : d_{X}(\gamma o, \eta o) \leq 4R + \lambda (n+w) \} \\
&\lesssim e^{(\delta(\Gamma) + 1)(4R + \lambda (n+w))} \\
&\lesssim e^{(\delta(\Gamma) + 1) \lambda n},
\end{align*} where the implicit constants in the above inequalities are independent of $n$. Therefore, there exists a constant $B = B(w) \geq 1$ such that, for all $n \geq 1$, the following holds: we can find a subset 
\begin{align}
\label{constructionoftheannulargenset}
\mathcal{D}_{\mathcal{C}, x,y,n,w,\epsilon} \subset \mathcal{A}_{\mathcal{C},x,y,n,w,\epsilon}
\end{align}
of cardinality
\begin{align}
\label{cardinalityestimate}
\# \mathcal{D}_{\mathcal{C},x,y,n,w,\epsilon} \geq \frac{1}{B} \cdot e^{-(\delta(\Gamma) + 1) \lambda n} \cdot \# \mathcal{A}_{\mathcal{C},x,y,n,w,\epsilon},  
\end{align}
such that
\begin{align}
\label{disjointusedinmainprop}
\mathcal{O}_{R}(o, \gamma_{1}o) \cap \mathcal{O}_{R}(o, \gamma_{2}o) = \emptyset \ \ \ \mathrm{for \ all} \ \ \ \gamma_{1}, \gamma_{2} \in \mathcal{D}_{\mathcal{C},x,y,n,w,\epsilon}.
\end{align}
We use the notation `$\mathcal{D}$' to emphasize that the elements of these sets have disjoint shadows. By using (\ref{annulardivergence}) and (\ref{cardinalityestimate}), we obtain the following estimate on the exponential growth rates of the norms of the Cartan projections of elements in these sets.
\begin{prop}
\label{annupoincdivwithdisjtshadows}
Let $w \geq 1$ be an arbitrary fixed integer. Then,
\begin{align*}
\limsup_{n \to \infty} \sum_{\gamma \in \mathcal{D}_{\mathcal{C},x,y,n,w,\epsilon}} e^{-\delta ||\kappa(\gamma)||} = \infty.
\end{align*}
\end{prop}
\begin{proof}
To simplify notation during the proof, we will set $\mathcal{A}_{n} := \mathcal{A}_{\mathcal{C},x,y,n,w,\epsilon}$ and $\mathcal{D}_{n} := \mathcal{D}_{\mathcal{C},x,y,n,w,\epsilon}$. Suppose to the contrary that 
\begin{align*}
\limsup_{n \to \infty} \sum_{\gamma \in \mathcal{D}_{n}} e^{-\delta ||\kappa(\gamma)||} < \infty.
\end{align*} This means that there exists a constant $M > 0$ so that, 
\begin{align}
\label{annupperbd}
\sum_{\gamma \in \mathcal{D}_{n}} e^{-\delta ||\kappa(\gamma)||} \leq M
\end{align} for all $n \geq 1$. Notice that, for all $\gamma \in \mathcal{D}_{n}$ and $\sigma \in \mathcal{A}_{n}$, we have
\begin{align*}
 ||\kappa(\gamma)|| &= ||\kappa(\sigma) - \kappa(\sigma) + \kappa(\gamma)|| \\
 &\leq ||\kappa(\sigma)|| + ||\kappa(\gamma) - \kappa(\sigma)|| \\
 &\leq ||\kappa(\sigma)|| + \lambda (n+w) \\
 &\leq ||\kappa(\sigma)|| + \lambda ||\kappa(\sigma)|| + \lambda w \\
&= (1+\lambda) ||\kappa(\sigma)|| + \lambda w.
\end{align*} Therefore,
\begin{align}
\label{relationofPoincareterms}
\notag -\delta_{0} ||\kappa(\sigma)|| &\leq \frac{-\delta_{0}}{1+\lambda} ||\kappa(\gamma)|| + \frac{\delta_{0}\lambda w}{1+\lambda} \\
&= -\delta ||\kappa(\gamma)|| - \Big( \frac{\delta_{0}}{1+\lambda} - \delta \Big) ||\kappa(\gamma)|| + \frac{\delta_{0} \lambda w}{1+\lambda}
\end{align} By (\ref{cardinalityestimate}), we have
\begin{align*}
\frac{\# \mathcal{A}_{n}}{\# \mathcal{D}_{n}} \leq B e^{(\delta(\Gamma)+1) \lambda n},
\end{align*} where the fixed constant $B$ is independent of $n$. Along with (\ref{annupperbd}) and (\ref{relationofPoincareterms}), this gives
\begin{align}
\label{upperboundPoincare}
\sum_{\sigma \in \mathcal{A}_{n}} e^{-\delta_{0}||\kappa(\sigma)||} &\leq B e^{\frac{\delta_{0} \lambda w}{1+\lambda}} \sum_{\gamma \in \mathcal{D}_{n}} e^{(\delta(\Gamma)+1)\lambda n - \Big( \frac{\delta_{0}}{1+\lambda} - \delta \Big) ||\kappa(\gamma)||} \cdot e^{-\delta ||\kappa(\gamma)||},
\end{align} for all $n \geq 1$. To bound the sum on the right uniformly in $n$, it remains to estimate the first exponential term inside the sum on the right-hand side. We have
\begin{align*}
(\delta(\Gamma)+1)\lambda n - \Big( \frac{\delta_{0}}{1+\lambda} - \delta \Big) ||\kappa(\gamma)|| &\leq (\delta(\Gamma)+1)\lambda n - \Big( \frac{\delta_{0}}{1+\lambda} - \delta \Big) n \\
&= \Bigg[(\delta(\Gamma)+1)\lambda - \Big( \frac{\delta_{0}}{1+\lambda} - \delta \Big)\Bigg] n \\
&< 0,
\end{align*} where the first inequality holds since $\gamma \in \mathcal{D}_{n}$ and the last inequality holds by (\ref{needed2}). Hence,
\begin{align*}
e^{(\delta(\Gamma)+1)\lambda n - \Big( \frac{\delta_{0}}{1+\lambda} - \delta \Big) ||\kappa(\gamma)||} \leq 1,   
\end{align*} for all $\gamma \in \mathcal{D}_{n}$ and all $n \geq 1$. By (\ref{annupperbd}), we see that (\ref{upperboundPoincare}) becomes
\begin{align*}
\sum_{\sigma \in \mathcal{A}_{n}} e^{-\delta_{0}||\kappa(\sigma)||} &\leq B e^{\frac{\delta_{0} \lambda w}{1+\lambda}} \sum_{\gamma \in \mathcal{D}_{n}} e^{-\delta ||\kappa(\gamma)||} \leq M B e^{\frac{\delta_{0} \lambda w}{1+\lambda}},
\end{align*} for all $n \geq 1$. But this contradicts (\ref{annulardivergence}), so must indeed have 
\begin{align*}
\limsup_{n \to \infty} \sum_{\gamma \in \mathcal{D}_{n}} e^{-\delta ||\kappa(\gamma)||} = \infty,
\end{align*} as desired.
\end{proof}

The sets $\mathcal{D}_{\mathcal{C},x,y,n,w,\epsilon}$ of Proposition \ref{annupoincdivwithdisjtshadows} almost have the desired properties to serve as generating sets for our semigroups. However, there is no guarantee yet that the semigroups they generate will be Zariski dense in $G$. Thus our present goal will be to carefully add an element of $\Gamma$ into each set $\mathcal{D}_{\mathcal{C},x,y,n,w,\epsilon}$ in such a way that the resulting sets still consist solely of elements with pairwise disjoint shadows, while having the further property of generating Zariski dense semigroups.
\par 
To make this idea rigorous, first fix some $0 < \epsilon' < \frac{1}{3} \epsilon$. Since $\Lambda(\Gamma)$ and $\Lambda(\Gamma)^{-}$ are perfect sets, we can find limit points $a \in \Lambda(\Gamma) \smallsetminus \{x\}$ and $b \in \Lambda(\Gamma)^{-} \smallsetminus \{y\}$ so that 
\begin{align}
\label{firstminimality} 2\epsilon' &< d(a,x) < \epsilon - \epsilon',  \ \ \ \mathrm{and} \\
\label{secondminimality} 0 &< d_{\mathrm{Haus}}(\mathcal{Z}_{b}, \mathcal{Z}_{y}) < \epsilon - \epsilon'.
\end{align} By (\ref{firstminimality}), we have $\overline{B_{\epsilon'}(a)} \cap \overline{B_{\epsilon'}(x)} = \emptyset$. Then using Lemma \ref{diametertozero} and an argument by contradiction using Proposition \ref{northsouthflagvars} (which we have already done several times), the following holds: there exists $n_{0} \geq 1$ so that for all $n \geq n_{0}$ and all $\gamma_{1} \in \Gamma_{\mathcal{C},x,y,n,\epsilon'}$ and $\gamma_{2} \in \Gamma_{\mathcal{C},a,b,n, \epsilon'}$, we have
\begin{align}
\label{disjointwhenlargeenough}
\mathcal{O}_{R}(o, \gamma_{1}o) \cap \mathcal{O}_{R}(o, \gamma_{2}o) = \emptyset.
\end{align} Notice that the inequalities (\ref{firstminimality}) and (\ref{secondminimality}) give
\begin{align}
\label{Cartansetswithdifferentlimitpoints}
\Gamma_{\mathcal{C},a,b,n,\epsilon'} \subset \Gamma_{\mathcal{C},x,y,n,\epsilon}.
\end{align}

Increasing $n_{0}$ if necessary, Proposition \ref{limitsetcrosslimitset3} tells us that every element of the set $\mathcal{D}_{\mathcal{C},x,y,n,w,\epsilon'} \subset \Gamma_{\mathcal{C},x,y,n,\epsilon}$ will be $2 \epsilon$-contracting. For each $n \geq n_{0}$, fix an element $\eta_{n} \in \mathcal{D}_{\mathcal{C},x,y,n,w,\epsilon'}$. Since $\eta_{n}$ is in particular loxodromic, it is regular semisimple. It follows that the centralizer of the Zariski closure of the subgroup $\langle \eta_{n} \rangle < \Gamma$ generated by $\eta_{n}$ is a maximal torus of $G$. We may therefore apply Proposition 4.4 of Tits' famous work \cite{Ti2} to conclude that the union of the Zariski closed and Zariski connected proper subgroups of $G$ that contain $\eta_{n}$ is contained in a proper Zariski closed subset $F_{\eta_{n}}$ of $G$. Denote by $F_{\eta_{n}}^{c}$ its complement in $G$, which is therefore a Zariski open subset of $G$. By Remark \ref{afortiori}, (\ref{disjointwhenlargeenough}), and (\ref{Cartansetswithdifferentlimitpoints}), there exists an element $\zeta_{n} \in \Gamma_{\mathcal{C},a,b,n,\epsilon'} \subset \Gamma_{\mathcal{C},x,y,n, \epsilon}$ so that
\begin{itemize}
\item[(1)] $\zeta_{n} \in F_{\eta_{n}}^{c}$, and
\item[(2)] $\mathcal{O}_{R}(o,\gamma o) \cap \mathcal{O}_{R}(o, \zeta o) = \emptyset$ for all $\gamma \in \mathcal{D}_{\mathcal{C},x,y,n,w,\epsilon'}$.
\end{itemize}
Now for each $n \geq n_{0}$, choose an integer $w_{n} \geq w$ so that $n \leq ||\kappa(\eta)|| \leq n+w_{n}$, and define
\begin{align}
\label{definitionofgeneratingset}
\mathcal{A}_{\mathcal{C},x,y,n,w_{n},\epsilon}' := \mathcal{D}_{\mathcal{C},x,y,n,w,\epsilon'} \cup \{\zeta_{n}\}.
\end{align}
By construction, we have $\mathcal{A}_{\mathcal{C},x,y,n,w_{n},\epsilon}' \subset \mathcal{A}_{\mathcal{C},x,y,n,w_{n},\epsilon}$. By item (1) above, the semigroup generated by $\mathcal{A}_{\mathcal{C},x,y,n,w_{n},\epsilon}'$ is Zariski dense in $G$. By item (2) above and the definition of $\mathcal{D}_{\mathcal{C},x,y,n,w,\epsilon'}$, we have
\begin{align}
\label{disjtshadowsgenset}
\mathcal{O}_{R}(o,\gamma_{1}o) \cap \mathcal{O}_{R}(o, \gamma_{2}o) = \emptyset \ \ \mathrm{for \ all} \ \ \gamma_{1}, \gamma_{2} \in \mathcal{A}_{\mathcal{C},x,y,n,w_{n},\epsilon}'.
\end{align}
Lastly,
\begin{align}
\label{annupoindivgenset}
\limsup_{n \to \infty} \sum_{\gamma \in \mathcal{A}_{\mathcal{C},x,y,n,w_{n},\epsilon}'} e^{-\delta ||\kappa(\gamma)||} = \infty,
\end{align} as an immediate consequence of Proposition \ref{annupoincdivwithdisjtshadows}. As we will see in the next section, these sets will serve as the generating sets of our semigroups.

\section{The Proof of Theorem \ref{MainTheorem1}}
\label{culminationsection}
In this section, we prove our main result, Theorem \ref{MainTheorem1}, which we restate below.
\begin{thm} [Theorem \ref{MainTheorem1}]
\label{MainTheorem2}
Let $G$ be a connected algebraic semisimple real Lie group with finite center and no compact factors, and let $\Gamma < G$ be a Zariski dense discrete subgroup. For every $0 < \delta < \delta (\Gamma)$, $\epsilon > 0$ sufficiently small, and $B \geq 1$, there exists a free, finitely generated subsemigroup $\Omega = \Omega_{\delta, \epsilon,B} \subset \Gamma$ with the following properties:
\begin{itemize}
    \item[(1)] Every element of $\Omega$ is either $\epsilon$-contracting or $2 \epsilon$-contracting.
    \item[(2)] The semigroup $\Omega$ is Zariski dense in $G$. 
    \item[(3)] The critical exponent of $\Omega$ satisfies 
    \begin{align*}
    \delta(\Omega) \geq \delta.
    \end{align*}
    \item[(4)] The semigroup $\Omega$ is $P$-Anosov. In fact,
    \begin{align*}
    \min_{\alpha \in \Delta} \alpha (\kappa (g)) \geq B |g|_{S},
    \end{align*} for all $g \in \Omega$, where $|\cdot|_{S}$ denotes the word-length with respect to the finite generating set $S$ that freely generates $\Omega$.
\end{itemize}
\end{thm}
We will deduce this theorem from Proposition \ref{keyproposition} and Proposition \ref{subadditivityofcartan} below, which in turn build on the results of the preceding sections.
\begin{prop}
\label{keyproposition}
For every $0 < \delta < \delta(\Gamma)$ and $\epsilon > 0$ sufficiently small, there exists a finite subset $S = S(\delta, \epsilon) \subset \Gamma$ so that the following holds: for any $m \geq 1$, if $\gamma \in S^{m}$, then $\gamma$ is either $\epsilon$-contracting or $2 \epsilon$-contracting and the set $\gamma \cdot S \subset S^{m+1}$ has the following properties:
\begin{itemize}
\item[(a)] If $\eta \in \gamma \cdot S$, then $\mathcal{S}_{2 \epsilon}(\eta) \subset \mathcal{S}_{4 \epsilon}(\gamma)$.
\item[(b)] The shadows $\{\mathcal{S}_{2 \epsilon}(\zeta)\}_{\zeta \in S}$ are pairwise disjoint.
\item[(c)] We have 
\begin{align*}
\sum_{\eta \in \gamma \cdot S} e^{-\delta ||\kappa(\eta)||} \geq e^{-\delta ||\kappa(\gamma)||}. 
\end{align*}
\end{itemize}
\end{prop}
\begin{proof}
Let $0 < \delta < \delta(\Gamma)$ be arbitrary, fix a pair $(x,y) \in \Lambda(\Gamma) \times \Lambda(\Gamma)^{-}$ such that $x \notin \mathcal{Z}_{y}$, and fix $0 < \epsilon < \frac{1}{8} d(x, \mathcal{Z}_{y})$. Let $r > 0$ be so that $\delta < \delta + r < \delta(\Gamma)$ and let $\mathcal{C} \subset \mathfrak{a}^{++}$ be an open cone so that, by Lemma \ref{limitsetcrosslimitset2}, we have 
\begin{align*}
\sum_{\gamma \in \Gamma_{\mathcal{C},x,y,n,\epsilon}} e^{-(\delta + r)||\kappa(\gamma)||} = \infty
\end{align*} for all $\epsilon > 0$ and $n \geq 1$. By Proposition \ref{limitsetcrosslimitset3} and Lemma \ref{contractsymspaceinclusion}, there exist constants $M = M(\epsilon/2) > 0$ and $R = R(\epsilon/2) > 0$ so that: if $n \geq M$ and $\gamma \in \Gamma_{\mathcal{C},x,y,n, \frac{\epsilon}{2}}$, then $\gamma$ is $\epsilon$-contracting and
\begin{align}
\label{shadowinclusiontogetdisjointness}
\mathcal{S}_{2 \epsilon}(\gamma) \subset \mathcal{S}_{\epsilon}(\gamma) \subset \mathcal{O}_{R}(o, \gamma o). 
\end{align} Let $\mathcal{A}_{\mathcal{C},x,y,n,w_{n},\frac{\epsilon}{2}}^{'}$ be as before (see (\ref{definitionofgeneratingset})). By (\ref{annupoindivgenset}), we have
\begin{align*}
\limsup_{n \to \infty} \sum_{\gamma \in \mathcal{A}_{\mathcal{C},x,y,n,w_{n},\frac{\epsilon}{2}}^{'}} e^{-\delta ||\kappa(\gamma)||} = \infty.
\end{align*} In particular, there exists $m_{0} \geq M$ so that, defining $S := \mathcal{A}_{\mathcal{C},x,y,m_{0},w_{m_{0}},\frac{\epsilon}{2}}^{'}$, we have
\begin{align}
\label{biggerthanone}
\sum_{\zeta \in S} e^{-\delta ||\kappa(\zeta)||} \geq 1.
\end{align} We claim that this set $S$ has all the properties in the statement of the proposition.
\par 
We first show that, for any $m \geq 1$, if $\gamma \in S^{m}$, then $\gamma$ is either $\epsilon$-contracting or $2 \epsilon$-contracting. By construction, every element of $S$ is $\epsilon$-contracting. Moreover, by Proposition \ref{limitsetcrosslimitset3}, for any pair of distinct elements $g \neq h \in S$ we have
\begin{align}
\label{gensetestimates}
d(x,x_{g}^{+}) < \epsilon, \ \ d_{\mathrm{Haus}}(\mathcal{Z}_{y}, \mathcal{Z}_{x_{h}^{-}}) < \epsilon, \ \ \mathrm{and} \ \ d(x,\mathcal{Z}_{y}) > 8 \epsilon,
\end{align} and therefore
\begin{align}
\label{gensetest2}
d(x_{g}^{+}, \mathcal{Z}_{x_{h}^{-}}) \geq d(x, \mathcal{Z}_{y}) - d(x,x_{g}^{+}) - d(\mathcal{Z}_{y}, \mathcal{Z}_{x_{h}^{-}}) > 6 \epsilon.
\end{align} Hence, by Proposition \ref{contractingsemigroup}, every element of the semigroup generated by $S$ is either $\epsilon$-contracting or $2 \epsilon$-contracting. We now prove the remaining assertions of the proposition.
\begin{proof} [Proof of Property (a)]
This is exactly the content of Proposition \ref{shadowinclusion}.
\end{proof}
\begin{proof} [Proof of Property (b)]
Since $S \subset \Gamma_{\mathcal{C},x,y,m_{0},\frac{\epsilon}{2}}$, the inclusions of (\ref{shadowinclusiontogetdisjointness}) give
\begin{align*}
\mathcal{S}_{2 \epsilon}(\zeta) \subset \mathcal{O}_{R}(o, \zeta o),
\end{align*} for all $\zeta \in S$. By the definition of the set $S$ (see (\ref{disjtshadowsgenset})), we conclude that the shadows $\mathcal{S}_{2 \epsilon}(\zeta_{1})$ and $\mathcal{S}_{2 \epsilon}(\zeta_{2})$ are disjoint for every pair of distinct elements $\zeta_{1}, \zeta_{2} \in S$, as desired.
\end{proof}
\begin{proof} [Proof of Property (c)]
Let $m \geq 1$ and $\gamma \in S^{m}$ be arbitrary. Let $\eta = \gamma \zeta$ for some $\zeta \in S$. Then 
\begin{align*}
||\kappa(\eta)|| \leq ||\kappa(\gamma)|| + ||\kappa(\zeta)||,
\end{align*}
hence
\begin{align*}
\sum_{\eta \in \gamma \cdot S} e^{-\delta ||\kappa(\eta)||} \geq e^{-\delta ||\kappa(\gamma)||} \sum_{\zeta \in S} e^{-\delta ||\kappa(\zeta)||} \geq e^{-\delta ||\kappa(\gamma)||},
\end{align*} where the last inequality follows from (\ref{biggerthanone}).
\end{proof}
This concludes the proof of all of the properties, and so also the proof of the proposition.
\end{proof}
The following results, which culminate in Proposition \ref{subadditivityofcartan}, will allow us to conclude that the Cartan projections of the elements of our semigroups are ``coarsely additive''. This is a key ingredient in establishing item (4) of Theorem \ref{MainTheorem2}, as well as in establishing the lower semicontinuity result Theorem \ref{lowersemicontinuity2}. 
\begin{lem}
\label{savedbyBruhat}
For every sufficiently small compact neighborhood $Q \subset G$ containing the identity, there exists a compact subset $Z = Z(Q) \subset \mathfrak{a}$ so that, for all $a_{1}, a_{2} \in A^{+}$, we have
\begin{align*}
\kappa(a_{1}Qa_{2}) - \kappa(a_{1}) - \kappa(a_{2}) \subset Z.
\end{align*}
\end{lem}
\begin{proof}
By the Bruhat decomposition, the set $N^{-}MAN$ is an open cell in $G$ containing the identity; see for instance Consequence 6 on page 139 of Knapp's book \cite{Kn2} for further details. Thus, for all sufficiently small compact neighborhoods $Q$ of the identity in $G$, we have $Q \subset N^{-}MAN$; hence, there exist compact neighborhoods of the identity $N^{-}_{Q} \subset N^{-}$, $A_{Q} \subset A$, and $N_{Q} \subset N$ so that $Q \subset N^{-}_{Q}MA_{Q}N_{Q}$.
\par 
Let $a_{1}, a_{2} \in A^{+}$ and $q \in Q$ be arbitrary. Write $q = n_{-}man$, where $n_{-} \in N^{-}_{Q}$, $m \in M$, $a \in A_{Q}$, and $n \in N_{Q}$. Recall that $M$ centralizes $A$ and normalizes $N$ and $N^{-}$. In particular, $m^{-1}n_{-}m =: n_{*}$ is an element of the compact subset 
\begin{align*}
(N_{Q}^{-})' := \bigcup_{g \in M} gN^{-}_{Q}g^{-1}
\end{align*}
of $N^{-}$. Thus,
\begin{align*}
\kappa(a_{1}n_{-}mana_{2}) &= \kappa(a_{1}m(m^{-1}n_{-}m)ana_{2}) \\
&= \kappa (ma_{1}n_{*}ana_{2}) \\
&= \kappa(a_{1}n_{*}ana_{2}) \\
&= \kappa \big( (a_{1} n_{*} a_{1}^{-1}) \cdot (a_{1} a a_{2}) \cdot (a_{2}^{-1} n a_{2}) \big)
\end{align*} We first show that $a_{1}n_{*}a_{1}^{-1} \in (N_{Q}^{-})'$.
\par 
Write $a_{1} = \exp(X)$ and $n_{*} = \exp(Y)$, where $X \in \mathfrak{a}^{+}$ and $Y \in \mathfrak{n}^{-}$. For each $\alpha \in \Sigma^{-}$, there exist unique $Y_{\alpha} \in \mathfrak{g}_{\alpha}$ so that
\begin{align*}
Y = \sum_{\alpha \in \Sigma^{-}} Y_{\alpha}.
\end{align*} Now
\begin{align}
\label{effectofconj}
a_{1}n_{*}a_{1}^{-1} = \exp(X) \exp(Y) \exp(-X) = \exp (\mathrm{Ad}(\exp(X))(Y)).
\end{align} But
\begin{align}
\label{actiononnilpotent}
\mathrm{Ad}(\exp(X))(Y) = (\exp(\mathrm{ad}X))(Y) = \sum_{\alpha \in \Sigma^{-}} (\exp(\mathrm{ad}(X))(Y_{\alpha}) = \sum_{\alpha \in \Sigma^{-}} e^{\alpha(X)}Y_{\alpha}.
\end{align} Since $X \in \mathfrak{a}^{+}$, we have $e^{\alpha(X)} \leq 1$ for all $\alpha \in \Sigma^{-}$. That is, $\mathrm{Ad}(\exp(X))$ has the effect of contracting each coordinate of $Y \in \mathfrak{n}^{-}$. Substituting the last equality of (\ref{actiononnilpotent}) into the last equality of (\ref{effectofconj}) shows that $a_{1}n_{*}a_{1}^{-1} \in (N_{Q}^{-})'$, as desired. An entirely analogous calculation shows that $a_{2}^{-1}na_{2} \in N_{Q}$.
\par 
By Lemma \ref{Yves'Estimate} applied to the compact subset $(N_{Q}^{-})' \cup N_{Q}$ of $G$, we obtain a compact subset $T \subset \mathfrak{a}$, depending only on $Q$, so that
\begin{align*}
\kappa \big( (a_{1} n_{*} a_{1}^{-1}) \cdot (a_{1} a a_{2}) \cdot (a_{2}^{-1} n a_{2}) \big) - \kappa(a_{1}aa_{2}) \in T.
\end{align*} Since $A_{Q}$ is compact, so is $\kappa(A_{Q}) \subset \mathfrak{a}$, hence 
\begin{align*}
\kappa(a_{1}aa_{2}) - \kappa(a_{1}) - \kappa(a_{2}) \in \kappa(A_{Q}).
\end{align*} Set $Z := T + \kappa(A_{Q}) \subset \mathfrak{a}$, which is compact and depends only on $Q$. Then the above calculations give
\begin{align*}
\kappa(a_{1}qa_{2}) - \kappa(a_{1}) - \kappa(a_{2}) \in Z.
\end{align*} Since $a_{1}, a_{2} \in A^{+}$ and $q \in Q$ were arbitrary, this concludes the proof.
\end{proof}
In the following lemma, we let $S$ denote a finite collection of $\frac{\epsilon}{2}$-contracting elements satisfying the hypotheses of Proposition \ref{contractingsemigroup}, and let $\Omega$ denote the positive semigroup it generates. 
\begin{lem}
\label{differenceofCartans}
For every $\epsilon > 0$ sufficiently small, there exists a compact subset $W = W(\epsilon) \subset \mathfrak{a}$ so that the following holds: for every $g \in \Omega$ and $h \in S$, we have
\begin{align*}
\kappa(gh) - \kappa(g) - \kappa(h) \in W.
\end{align*} 
\end{lem}
\begin{proof}
As in the proof of Lemma \ref{JordanCartanforBenoistsset}, there is a compact set $Q \subset G$, depending only on $\epsilon > 0$, so that for all $\gamma \in \Omega$, there exist $q_{\gamma} \in Q$, $m_{\gamma} \in M$, and $a_{\gamma} \in A$ so that
\begin{align}
\label{nicedecomposition}
\gamma = q_{\gamma} m_{\gamma} a_{\gamma} q_{\gamma}^{-1}.
\end{align} Observe further that, provided $\epsilon > 0$ is small enough, the set $Q^{-1}Q := \{q_{1}^{-1}q_{2} : q_{1},q_{2} \in Q\}$ is a compact neighborhood of the identity in $G$ to which Lemma \ref{savedbyBruhat} can be applied.

Let $g \in \Omega$ and $h \in S$ be arbitrary. Decomposing the elements $gh$, $g$, and $h$ as in (\ref{nicedecomposition}), we obtain
\begin{align*}
q_{gh}m_{gh}a_{gh}q_{gh}^{-1} = gh = q_{g}m_{g}a_{g}q_{g}^{-1}q_{h}m_{h}a_{h}q_{h}^{-1},
\end{align*} and therefore (recall that $M$ centralizes $A$)
\begin{align*}
a_{gh} = \big(m_{gh}^{-1}q_{gh}^{-1}q_{g}m_{g} \big) \cdot \big( a_{g}q_{g}^{-1}q_{h}a_{h} \big) \cdot \big( m_{h}q_{h}^{-1}q_{gh} \big)
\end{align*}
By Lemma \ref{Yves'Estimate}, there is compact set $W_{1} \subset \mathfrak{a}$, depending only on the compact subset $MQ^{-1}QM$ of $G$, hence only on $\epsilon > 0$, so that 
\begin{align}
\label{firstcartan}
\kappa(a_{gh}) - \kappa(a_{g}q_{g}^{-1}q_{h}a_{h}) \in W_{1}.
\end{align}
By Lemma \ref{savedbyBruhat}, there exists a compact subset $W_{2} \subset \mathfrak{a}$, depending only on $Q^{-1}Q$, hence only on $\epsilon > 0$, so that
\begin{align}
\label{secondcartan}
\kappa(a_{g}q_{g}^{-1}q_{h}a_{h}) - \kappa(a_{g}) - \kappa(a_{h}) \in W_{2}.
\end{align}
Lastly, by Lemma \ref{JordanCartanforBenoistsset}, there is a compact set $W_{3} \subset \mathfrak{a}$, depending only on $\epsilon > 0$, so that 
\begin{align}
\label{thirdcartan}
\kappa(gh) - \kappa(a_{gh}) \in W_{3}, \ \ \kappa(a_{g}) - \kappa(g) \in W_{3}, \ \ \mathrm{and} \ \ \kappa(a_{h}) - \kappa(h) \in W_{3}.
\end{align}
Combining (\ref{firstcartan}), (\ref{secondcartan}), and (\ref{thirdcartan}), we obtain
\begin{align*}
\kappa(gh) - \kappa(g) - \kappa(h) \in W_{1} + W_{2} + 3W_{3}.
\end{align*} Since $g \in \Omega$ and $h \in S$ were arbitrary, it follows that the lemma holds with the compact subset $W := W_{1} + W_{2} + 3 W_{3} \subset \mathfrak{a}$.
\end{proof}
\begin{prop}
\label{subadditivityofcartan}
With $W \subset \mathfrak{a}$ as in Lemma \ref{differenceofCartans}, the following holds: for all $l \geq 1$ and $g_{1}, \dots, g_{l} \in S$, we have
\begin{align*}
\kappa(g_{1} \cdots g_{l}) - \sum_{i=1}^{l} \kappa(g_{i}) \in (l-1) \cdot W.
\end{align*}
\end{prop}
\begin{proof}
This follows by inductively applying Lemma \ref{differenceofCartans}. 
\end{proof}

We now give the proof of Theorem \ref{MainTheorem2}.
\begin{proof} [Proof of Theorem \ref{MainTheorem2}]
Given $0 < \delta < \delta(\Gamma)$, $\epsilon > 0$ sufficiently small, and $B \geq 1$, first let 
\begin{align*}
S = \mathcal{A}_{\mathcal{C},x,y,m_{0},w_{m_{0}}, \frac{\epsilon}{2}}^{'}
\end{align*} be the finite subset of $\Gamma$ furnished by Proposition \ref{keyproposition}. Let $W = W(\epsilon/2) \subset \mathfrak{a}$ be the compact subset furnished by Lemma \ref{differenceofCartans}, and set
\begin{align}
\label{defnofC0}
C_{0} := \max \{||x|| :  x \in W\}.
\end{align}
As $C_{0}$ depends only on $\epsilon$, we may further assume that $m_{0}$ is sufficiently large so that
\begin{align}
\label{Cartanisbigenough}
\min_{\alpha \in \Delta} \min_{g \in S} \alpha(\kappa(g)) > \Big( C_{0} \cdot \max_{\alpha \in \Delta}  ||\alpha||_{\mathrm{op}} \Big) + B.
\end{align} Now let 
\begin{align*}
\Omega = \Omega_{\delta, \epsilon, B} := \bigcup_{m=1}^{\infty} S^{m}
\end{align*} be the semigroup generated by $S$. By Proposition \ref{keyproposition}, we know if $\gamma \in S^{m}$, $m \geq 1$, then $\gamma$ is either $\epsilon$-contracting or $2 \epsilon$-contracting. This establishes item (1) in the statement of the theorem. That the semigroup $\Omega$ is Zariski dense in $G$ follows by the construction of $S$ (see the discussion above (\ref{disjtshadowsgenset})). This establishes item (2) in the statement of the theorem.
\par 
We now show that $\Omega$ is in fact a free semigroup, being freely generated by $S$. Suppose there exists some element $\gamma \in \Omega$ which can be written as 
\begin{align*}
\gamma = g_{1} \cdots g_{k} = h_{1} \cdots h_{j},
\end{align*} where $g_{i},h_{l} \in S$ for all $1 \leq i \leq k$ and $1 \leq l \leq j$. We need to show that $k = j$ and $g_{i} = h_{i}$ for all $1 \leq i \leq k=j$. There are two cases to consider:
\par 
$\textbf{Case 1.}$ In this case, we have $k = j$. There exists a largest integer $0 \leq m \leq k$
so that $g_{i} = h_{i}$ for all $0 \leq i \leq m$ (where we define $g_{0} := h_{0} := \mathrm{id}$). Suppose that $m < k$ (as otherwise we are done). Then
\begin{align}
\label{sameword}
g_{m+1} \cdots g_{k} = h_{m+1} \cdots h_{j}
\end{align} and $g_{m+1} \neq h_{m+1}$. Repeatedly applying property (a) of Proposition \ref{keyproposition}, we obtain
\begin{align*}
\mathcal{S}_{2 \epsilon}(g_{m+1} \cdots g_{k-1} g_{k}) \subset \mathcal{S}_{4 \epsilon}(g_{m+1} \cdots g_{k-1}) \subset \mathcal{S}_{2 \epsilon}(g_{m+1} \cdots g_{k-1}) \subset \cdots \subset \mathcal{S}_{2 \epsilon}(g_{m+1}),
\end{align*} and likewise
\begin{align*}
\mathcal{S}_{2 \epsilon}(h_{m+1} \cdots h_{j}) \subset \mathcal{S}_{2 \epsilon}(h_{m+1}).
\end{align*} But since $g_{m+1} \neq h_{m+1}$, property (b) of Proposition \ref{keyproposition} implies that 
\begin{align*}
\mathcal{S}_{2 \epsilon}(g_{m+1}) \cap \mathcal{S}_{2 \epsilon}(h_{m+1}) = \emptyset,
\end{align*} hence also
\begin{align*}
\mathcal{S}_{2 \epsilon}(g_{m+1} \cdots g_{k}) \cap \mathcal{S}_{2 \epsilon}(h_{m+1} \cdots h_{j}) = \emptyset.
\end{align*} This contradicts (\ref{sameword}).
\par 
\textbf{Case 2.} It remains to consider the case when $k \neq j$. Without loss of generality, assume $k < j$. There exists a smallest integer $m \geq 1$ so that $g_{m} \neq h_{m}$. If $m < k$, then we can argue as in the first case. If $g_{i} = h_{i}$ for all $1 \leq i \leq k$, we obtain
\begin{align*}
\mathrm{id} = h_{k+1} \cdots h_{j}.
\end{align*} But this is impossible, since the product $h_{k+1} \cdots h_{j}$ is either $\epsilon$-contracting or $2 \epsilon$-contracting; in particular, it is a loxodromic element, whereas the identity element is not. This concludes the second case, and therefore also the proof that $\Omega$ is a free subsemigroup of $\Gamma$, freely generated by the set $S \subset \Gamma$.
\par 
We now show that $\delta(\Omega) \geq \delta$, which will establish item (3) of the theorem. Inductively applying property (c) of Proposition \ref{keyproposition}, we obtain
\begin{align*}
\sum_{\eta \in g \cdot S^{m}} e^{-\delta ||\kappa(\eta)||} \geq e^{-\delta ||\kappa(g)||} > 0.
\end{align*} for all $g \in S$ and $m \geq 1$. Since $\Omega$ is a free semigroup, we compute
\begin{align*}
\sum_{\gamma \in \Omega} e^{-\delta ||\kappa(\gamma)||} &= \sum_{m=0}^{\infty} \sum_{g \in S} \sum_{\eta \in g \cdot S^{m}} e^{-\delta ||\kappa(\eta)||} \geq \sum_{m=0}^{\infty} \bigg( \sum_{g \in S} e^{-\delta ||\kappa (g)||} \bigg) = \infty.
\end{align*} Hence $\delta(\Omega) \geq \delta$, as desired.
\par 
It remains to establish item (4). Since $\Omega$ is a free subsemigroup, it suffices to show that for all $m \geq 1$ and any collection $g_{1}, \dots, g_{m} \in S$, we have
\begin{align}
\label{toshowAnosovness}
\min_{\alpha \in \Delta} \alpha(\kappa(g_{1} \cdots g_{m})) \geq Bm.
\end{align} Define
\begin{align*}
C := \min_{\alpha \in \Delta} \min_{g \in S} \alpha(\kappa(g)) - C_{0} \cdot \max_{\alpha \in \Delta}  ||\alpha||_{\mathrm{op}},
\end{align*} and notice that $C > B$ by (\ref{Cartanisbigenough}). By Proposition \ref{subadditivityofcartan}, we have
\begin{align*}
\Big| \Big| \kappa(g_{1} \cdots g_{m}) - \sum_{i=1}^{m} \kappa(g_{i}) \Big| \Big| \leq (m-1) C_{0},
\end{align*} and therefore, for any $\alpha \in \Delta$,
\begin{align*}
\alpha(\kappa(g_{1} \cdots g_{m})) &\geq \bigg( \sum_{i=1}^{m} \alpha (\kappa(g_{i})) \bigg) - (m-1)C_{0} ||\alpha||_{\mathrm{op}} \\
&\geq \Big( \min_{g \in S} \alpha(\kappa(g)) - C_{0} ||\alpha||_{\mathrm{op}} \Big) m \\
&\geq Cm \\
&> Bm.
\end{align*} This verifies (\ref{toshowAnosovness}), which concludes the proof of the theorem.
\end{proof}
\section{Lower Semicontinuity of the Critical Exponent}
\label{lowersemicontsection}
In this section, we apply Theorem \ref{MainTheorem1} to prove Theorem \ref{lowersemicontinuity1}, which establishes that the critical exponent is lower semicontinuous in the Chabauty topology if the Chabauty limit of a sequence of Zariski dense discrete subgroups is again a Zariski dense discrete subgroup. Precisely, we prove the following:
\begin{thm} [Theorem \ref{lowersemicontinuity1}]
\label{lowersemicontinuity2}
Let $G$ be a connected algebraic semisimple real Lie group with finite center and no compact factors. Let $\{\Gamma_{n}\}$ be a sequence of Zariski dense discrete subgroups of $G$ which converges in the Chabauty topology to a Zariski dense discrete subgroup $\Gamma$ of $G$. Then,
\begin{align*}
\liminf_{n \to \infty} \delta(\Gamma_{n}) \geq \delta(\Gamma).
\end{align*}
\end{thm}
We consider the space Sub$(G)$ of all closed subgroups of $G$ equipped with the \emph{Chabauty topology}. Precisely, a sequence of closed subgroups $H_{n}$ of $G$ converges to a subgroup $H$ of $G$ as $n \to \infty$ in the Chabauty topology if, for any $h \in H$, there exists a sequence $h_{n} \in H_{n}$ with $h_{n} \rightarrow h$ and, moreover, the limit points of any sequence of elements $g_{n} \in H_{n}$ belong to $H$. If a sequence of closed subgroups $H_{n}$ converges to a closed subgroup $H$, then we say that $H$ is the \emph{Chabauty limit of the} $H_{n}$. We remark that the Chabauty limit of a sequence of discrete subgroups need \emph{not} be a discrete subgroup. For more background and interesting applications of the Chabauty topology, we refer the reader to \cite{ABBGNRS}, \cite{FHR}, and \cite{FO}. We now give the proof of Theorem \ref{lowersemicontinuity2}.
\begin{proof} [Proof of Theorem \ref{lowersemicontinuity2}]
Let $\{\Gamma_{n}\} \subset \mathrm{Sub}(G)$ be a sequence of Zariski dense discrete subgroups of $G$ which converges in the Chabauty topology to a Zariski dense discrete subgroup $\Gamma$ of $G$. We will show that 
\begin{align*}
\liminf_{n \to \infty} \delta(\Gamma_{n}) \geq \delta(\Gamma)
\end{align*} by showing that, for all $0 < \delta_{0} < \delta(\Gamma)$ and any subsequence $\{\Gamma_{n_{k}}\} \subset \{\Gamma_{n}\}$, there exist free subsemigroups $\Omega_{n_{k}} \subset \Gamma_{n_{k}}$ with $\delta(\Omega_{n_{k}}) \geq \delta_{0}$. 

Let $0 < \delta_{0} < \delta(\Gamma)$ be arbitrary and fix a transverse pair $(x,y) \in \Lambda(\Gamma) \times \Lambda(\Gamma)^{-}$. Fix $0 < \epsilon < \frac{1}{8} d(x, \mathcal{Z}_{y})$ and let $C_{0} := C_{0}(\epsilon)$ be the constant defined in (\ref{defnofC0}). Now fix $0 < \delta_{0} < \delta < \delta(\Gamma)$ and set $B := \frac{3 \delta_{0} C_{0}}{\delta - \delta_{0}}$. By Theorem \ref{MainTheorem2} and (\ref{gensetest2}), there exists a finitely generated free semigroup $\Omega \subset \Gamma$ with finite generating set $S \subset \Gamma$ so that 
\begin{itemize}
    \item[(a)] $\sum_{\gamma \in \Omega} e^{-\delta ||\kappa(\gamma)||} = \infty$,
    \item[(b)] $\min_{\alpha \in \Delta} \alpha(\kappa(\gamma)) \geq B |\gamma|_{S} 
    \ \Rightarrow \ ||\kappa(\gamma)|| \geq B|\gamma|_{S}$, for all $\gamma \in \Omega$, and 
    \item[(c)] Each element of $S$ is $\epsilon$-contracting and moreover, for any two distinct elements $g,h \in S$, we have $d(x_{g}^{+}, \mathcal{Z}_{x_{h}^{-}}) > 6 \epsilon$.
\end{itemize}
Since the sequence $\{\Gamma_{n}\}$ Chabauty converges to $\Gamma$, so does any subsequence $\{\Gamma_{n_{k}}\}$. Hence it suffices to show that, for all $n \geq 1$ sufficiently large, we can find free subsemigroups $\Omega_{n} \subset \Gamma_{n}$ with $\delta(\Omega_{n}) \geq \delta_{0}$. Write $S = \{\omega_{1}, \dots, \omega_{l}\}$. By definition, there exist sequences $\{\omega_{n,j}\}_{n \geq 1}$, $1 \leq j \leq l$, with $\omega_{n,j} \in \Gamma_{n}$ such that $\omega_{n,j} \rightarrow \omega_{j}$ as $n \to \infty$. In particular, we have 
\begin{align}
\label{need1}
||\kappa(\omega_{n,j}) - \kappa(\omega_{j})|| \leq C_{0},
\end{align} for all $1 \leq j \leq l$ and all $n$ large enough. Furthermore, for all $n$ sufficiently large and all $1 \leq j \leq l$, the elements $\omega_{n,j}$ are all loxodromic. Moreover, for all $j$, 
\begin{align*}
d(x_{\omega_{n,j}}^{+}, x_{\omega_{j}}^{+}) \rightarrow 0 \ \ \ \mathrm{and} \ \ \ d_{\mathrm{Haus}} \big( \mathcal{Z}_{x_{\omega_{n,j}}^{-}}, \mathcal{Z}_{x_{\omega_{j}}^{-}} \big) \rightarrow 0 \ \ \ \mathrm{as} \ \ \ n \to \infty.
\end{align*} Hence, for $n \geq 1$ sufficiently large, the sets $S_{n} := \{\omega_{n,1}, \dots, \omega_{n,l} \} \subset \Gamma_{n}$ consist of $\epsilon$-contracting elements such that
\begin{align*}
d \big(x_{\omega_{n,i}}^{+}, \mathcal{Z}_{x_{\omega_{n,j}}^{-}} \big) \geq 6 \epsilon,
\end{align*}
for $1 \leq i \neq j \leq l$. By Proposition \ref{contractingsemigroup} and  Proposition \ref{shadowinclusion}, every element of the subsemigroup $\Omega_{n}$ of $\Gamma_{n}$ generated by the finite set $S_{n}$ is either $\epsilon$-contracting or $2 \epsilon$-contracting, and, if $m \geq 1$ and $\eta = \gamma \zeta$ where $\gamma \in S_{n}^{m}$ and $\zeta \in S_{n}$, then $\mathcal{S}_{2 \epsilon}(\eta) \subset \mathcal{S}_{4 \epsilon}(\gamma)$. Furthermore, since $\omega_{n,j} \rightarrow \omega_{j}$ as $n \to \infty$ and $\mathcal{Z}_{x_{\omega_{n,j}}^{-}} \rightarrow \mathcal{Z}_{x_{\omega_{j}}^{-}}$ in the Hausdorff topology as $n \to \infty$, we see that $\mathcal{S}_{2 \epsilon}(\omega_{n,j}) \to \mathcal{S}_{2 \epsilon}(\omega_{j})$ in the Hausdorff topology. Hence, for $n \geq 1$ sufficiently large, we can additionally guarantee that the shadows $\{\mathcal{S}_{2 \epsilon}(\omega_{n,j})\}_{\omega_{n,j} \in S_{n}}$ are pairwise disjoint. Then, arguing just as in the proof of Theorem \ref{MainTheorem2}, we see that $\Omega_{n} \subset \Gamma_{n}$ is a free subsemigroup, freely generated by the finite set $S_{n}$. 

We now show that each such semigroup $\Omega_{n}$ satisfies $\delta(\Omega_{n}) \geq \delta_{0}$. Let $\Omega' := \Omega_{n}$ be an arbitrary such semigroup and denote by $g' \in \Omega'$ the element corresponding to $g \in \Omega$ under the natural homomorphism of free semigroups $\Omega \rightarrow \Omega'$ defined by $\omega_{j} \mapsto \omega_{n,j}$ for $1 \leq j \leq l$. By construction, the conclusion of Proposition \ref{subadditivityofcartan} holds for the semigroup $\Omega'$. Hence Proposition \ref{subadditivityofcartan}, (\ref{need1}), and the triangle inequality yield
\begin{align}
\label{boundcartanofsemigroups}
||\kappa(g) - \kappa(g')|| \leq 2 (|g|_{S} - 1)C_{0} + |g|_{S}C_{0} < 3C_{0}|g|_{S}.
\end{align}
Thus,
\begin{align*}
||\kappa(g')|| \leq ||\kappa(g)|| + 3C_{0}|g|_{S},
\end{align*} so, by item (b) above, we obtain
\begin{align*}
-\delta_{0} ||\kappa(g')|| &\geq -\delta_{0}||\kappa(g)|| - \delta_{0} \cdot 3 C_{0}|g|_{S} \\
&= -\delta ||\kappa(g)|| + (\delta - \delta_{0})||\kappa(g)|| - 3 \delta_{0}C_{0}|g|_{S} \\
&\geq -\delta ||\kappa(g)|| + (\delta - \delta_{0})B |g|_{S} - 3 \delta_{0}C_{0}|g|_{S} \\
&= -\delta ||\kappa(g)|| + 3 \delta_{0}C_{0}|g|_{S} - 3 \delta_{0}C_{0}|g|_{S} \\
&= -\delta ||\kappa(g)||.
\end{align*}
Along with item (a) above, we conclude that 
\begin{align*}
\sum_{g' \in \Omega'} e^{-\delta_{0}||\kappa(g')||} \geq \sum_{g \in \Omega} e^{-\delta ||\kappa(g)||} = \infty,
\end{align*} and therefore $\delta(\Omega') \geq \delta_{0}$. This concludes the proof. 
\end{proof} 
As an application of this theorem, we provide a new proof of a result concerning the Chabauty convergence of infinite covolume Zariski dense discrete subgroups in semisimple Lie groups with Kazhdan's property (T). The author's advisors Andrew Zimmer and Sebastian Hurtado-Salazar informed him that the following result was already known. However, the proof that Hurtado-Salazar informed the author of relies on the deep superrigidity theorems of Margulis \cite{Mar} (in the higher rank case) and of Corlette \cite{Cor2} and Gromov-Schoen \cite{GS} (in the exotic rank one cases). Our proof, by contrast, only relies on Theorem \ref{lowersemicontinuity2} and Leuzinger's Theorem \ref{Leuzingergapphenomenon}, providing a new, simpler, proof.
\begin{cor}
\label{Infinitecovolumeconvergence}
Let $G$ be a connected semisimple real Lie group with finite center, no compact factors, and with Kazhdan's property (T). Let $\{\Gamma_{n}\}$ be a sequence of infinite covolume Zariski dense discrete subgroups of $G$ which converges in the Chabauty topology to a Zariski dense discrete subgroup $\Gamma$ of $G$. Then $\Gamma$ also has infinite covolume in $G$.
\end{cor}
\begin{proof}
Since $G$ has property (T) and each discrete subgroup $\Gamma_{n}$ has infinite covolume in $G$, Leuzinger's Theorem \ref{Leuzingergapphenomenon} implies that there is a constant $\epsilon = \epsilon(G) > 0$ so that 
\begin{align*}
h_{\mathrm{vol}}(X) - \epsilon \geq \delta (\Gamma_{n}),  
\end{align*} for all $n \geq 1$. By Theorem \ref{lowersemicontinuity2}, we obtain
\begin{align*}
h_{\mathrm{vol}}(X) - \epsilon \geq \liminf_{n \to \infty} \delta(\Gamma_{n}) \geq \delta(\Gamma).
\end{align*} Hence $\Gamma$ must have infinite covolume in $G$.
\end{proof}

\end{document}